\documentclass[12pt]{article}
 \usepackage{amsfonts,latexsym,amsmath,amssymb,amsthm}
\usepackage[mathscr]{eucal}
\usepackage{graphicx,color}
\usepackage{comment}
\usepackage[hidelinks]{hyperref}
\usepackage{epstopdf}
\usepackage[margin = 2cm]{geometry}
\usepackage[inline]{enumitem}
\usepackage{mathtools}
\usepackage{cite}
\usepackage{xcolor}
\usepackage{cleveref}
\usepackage{tikz}
\usepackage{array}
\usepackage{calc}
\usepackage[normalem]{ulem}
\usepackage{stmaryrd}
\usepackage{bbm}
\usepackage{pdfpages}

\newcommand{\J}{\mathcal J}
\newcommand{\F}{\mathbb F}
\newcommand{\Ps}{\Psi}

\usepackage{setspace}
\usetikzlibrary{arrows, arrows.meta}

\newtheorem{theorem}{Theorem}
\newtheorem{lemma}{Lemma}
\newtheorem{proposition}{Proposition}

\theoremstyle{definition}
\newtheorem{definition}{Definition}
\newtheorem{assumption}{Assumption}

\newtheorem{remarkx}{Remark}

\newtheorem{examplex}{Example}
\newenvironment{remark}
  {\pushQED{\qed}\remarkx}
  {\popQED\endremarkx}
\newenvironment{example}
  {\pushQED{\qed}\examplex}
  {\popQED\endexamplex}

\DeclarePairedDelimiter{\norm}{\lVert}{\rVert}

\DeclareMathOperator{\diff}{d\!} % Differential
\DeclareMathOperator*{\esssup}{ess\,sup}

\newcommand{\suchthat}{\ifnum\currentgrouptype=16 \mathrel{}\middle|\mathrel{}\else\mid\fi}
\newcommand{\Hscr} {H}

\def\sat{\mathrm{sat}}

\makeatletter
\newcommand{\customlabel}[2]{%
   \protected@write \@auxout {}{\string \newlabel {#1}{{#2}{\thepage}{#2}{#1}{}} }%
   \hypertarget{#1}{#2}
}

\makeatother

\begin{document}

\setlist[enumerate]{label={(\alph*)}}
\setlist[enumerate, 2]{label={(\alph{enumi}-\roman*)}}
\newlist{listhypo}{enumerate}{1}
\setlist[listhypo]{label={\textup{(H\arabic*)}}, ref={\textup{(H\arabic*)}}}

\title{\bf Semi-uniform stability estimates for impedance passive systems with saturated feedback\footnotemark[1]}

\author{Swann Marx\footnotemark[2]}

\maketitle

\footnotetext[1]{LS2N, \'Ecole Centrale de Nantes \& CNRS UMR 6004, F-44000 Nantes, France. E.mail: swann.marx@ls2n.fr}
\footnotetext[2]{This work has been supported by the Agence National de la Recherche (ANR) under the ROTATION project (grant no. ANR-24-CE48-0759).}

\begin{abstract}
This article investigates the long-time behavior of (possibly infinite-dimensional) impedance passive systems under saturated output feedback and external disturbances. We assume that, in the absence of saturation and disturbances, the underlying linear output feedback exponentially stabilizes the system. Our main contribution is to show that fractional Sobolev regularity of the free output is preserved by the nonlinear feedback. More precisely, if the initial condition belongs to a suitable interpolation space associated with the linear closed-loop generator, then both the output and the state of the nonlinear closed-loop system inherit the corresponding fractional regularity. This regularity is sufficiently weak to be shared by the linear and nonlinear closed-loop systems, thereby avoiding the identification of the nonlinear generator domain. Combining this regularity result with observability estimates for the linear system yields a characterization of the asymptotic behavior of the nonlinear closed-loop system in the presence of disturbances. In particular, under the impedance passivity framework and exact observability and regularity assumptions, we establish a semi-uniform input-to-state stability property. The theory is illustrated by a multidimensional wave equation with nonlinear boundary damping.
\end{abstract}

\section{Introduction}

Passive systems constitute one of the cornerstones of modern control theory. Their intrinsic energy balance provides a natural framework for the analysis of stability and robustness, while allowing the treatment of a large variety of distributed parameter systems through abstract operator-theoretic methods. In particular, the theory of regular linear systems introduced and studied in
\cite{weiss1994regular,staffans2005well,tucsnak2014well,weiss2003get}
provides a unified framework for boundary control systems involving unbounded control and observation operators.

This paper is concerned with the long-time behavior of impedance passive systems interconnected with a static monotone feedback law and subject to external disturbances. We assume that, in the absence of disturbances and nonlinear effects, the corresponding linear output feedback exponentially stabilizes the system
\cite{liu1997locally,curtain2006exponential,curtain2019strong}; see also
\cite{weiss2000optimizability}
for related stabilizability issues. Although our motivating application is actuator saturation, our analysis applies to a broader class of monotone passive nonlinearities.

A major difficulty in the stability analysis of nonlinear passive feedback systems stems from the mismatch between the quantities naturally arising in observability estimates and those appearing in nonlinear dissipation inequalities. Indeed, exact observability is usually formulated in the $L^2$ framework, whereas monotone feedback laws naturally generate $L^1$-type dissipation terms. Bridging this gap therefore requires additional regularity of the output. The philosophy of the present paper is that fractional Sobolev regularity provides precisely this missing functional-analytic ingredient: it allows one to convert quadratic observability estimates into the improved integrability properties naturally required in the analysis of nonlinear passive feedback systems.

The main contribution of this work is to show that the fractional Sobolev regularity of the free output is preserved by monotone passive feedback interconnections. More precisely, assuming that the initial state belongs to a suitable interpolation space associated with the linear closed-loop generator, we prove that both the nonlinear output and the corresponding nonlinear state inherit the same fractional regularity. Remarkably, the required regularity is strictly weaker than the nonlinear generator domain and is therefore shared by both the linear and nonlinear closed-loop systems. Consequently, our analysis avoids any explicit characterization of the nonlinear generator and relies only on interpolation spaces naturally associated with the linear closed-loop problem.

This regularity result is of independent interest. Combined with Sobolev embeddings, it yields improved integrability properties of the nonlinear output, which in turn allow observability estimates for the underlying linear system to be transferred to the nonlinear closed-loop system. Under the sole assumption that the open-loop system is impedance passive, we consequently establish a semi-uniform input-to-state stability property with respect to external disturbances.

The derivation of ISS estimates for saturated infinite-dimensional systems is considerably more delicate than in finite dimension. Indeed, several works
\cite{jacob2020remarks,chitour2020one}
have shown that saturation nonlinearities may destroy uniform decay properties. Although asymptotic convergence is preserved, decay estimates generally depend on the size of the initial condition. Such a loss of global uniformity is a genuinely infinite-dimensional phenomenon and already appears for linear systems exhibiting non-uniform decay. This strongly motivates the search for stability notions weaker than uniform ISS, such as the semi-uniform ISS property introduced in
\cite{wakaiki2022semi}.

The stabilization of passive systems has been extensively studied over the past decades. For abstract linear systems with collocated feedback, representative references include
\cite{liu1997locally}
for bounded control operators and
\cite{curtain2006exponential}
for unbounded ones. More generally, the theory of regular linear systems
\cite{weiss1994regular,staffans2005well,tucsnak2009observation}
provides a complete framework for the well-posedness of feedback interconnections involving admissible control and observation operators.

In the nonlinear setting, stabilization under saturation constraints has been extensively investigated for particular partial differential equations, including wave equations
\cite{alabau1999stabilisation,alabau2002indirect,Alabau2012Recent,chitour2020one,chitour2019p,pierre2000strong,vanspranghe2021velocity,Zuazua1990Uniform}
and the Korteweg--de Vries equation
\cite{mcpa2017siam,parada2022global}. More generally, saturation has long been recognized as a fundamental issue in control theory
\cite{sussmann1991saturation,tarbouriech2011book_saturating,teel1992globalsaturation}. Most of these contributions rely on PDE-specific techniques such as multiplier methods, hidden regularity, or Lyapunov functionals. By contrast, the present work is formulated entirely within the abstract framework of regular linear systems and therefore applies simultaneously to a broad class of distributed parameter systems.

Our approach is also closely related to, yet fundamentally different from, two classical directions in the literature.

On the one hand, the seminal work
\cite{tataru1998boundary}
established hidden regularity properties for the boundary traces of wave equations by combining microlocal and elliptic techniques. Such results reveal additional regularity that is intrinsic to the underlying partial differential equation. In contrast, we do not establish new trace regularity results. Instead, assuming that the free output already belongs to a suitable fractional Sobolev space---a property which, in many relevant situations, follows from interpolation estimates for regular linear systems
\cite{paunonen2024admissibility,guiver2024mixed,guiver2024operator}---we prove that this regularity is preserved by monotone passive feedback laws. Consequently, our analysis is independent of the underlying partial differential equation and applies simultaneously to arbitrary regular impedance passive systems.

On the other hand, our stability analysis differs substantially from the classical energy methods in
\cite{komornik1994exact,komornik1990stabilization}; see also the convexity framework introduced in
\cite{alabau2005convexity}. Those approaches derive decay estimates directly from energy identities, multiplier techniques and convexity arguments. Here, we instead exploit the propagation of fractional Sobolev regularity to obtain improved integrability properties of the nonlinear output through Sobolev embeddings. These estimates are then combined with observability inequalities in order to transfer stability properties from the linear closed-loop system to the nonlinear one. We therefore view the present work as complementary to the classical energy approach, providing a purely abstract input--output framework for nonlinear passive feedback systems.

Among abstract semigroup approaches, the Lyapunov framework developed in
\cite{marx2018stability}
provides stability results for passive systems with bounded control operators. However, many boundary control systems fall outside this setting because the associated control and observation operators are unbounded. The present work removes this restriction by working directly at the input--output level. Rather than constructing Lyapunov functionals, our analysis relies on a preservation principle for fractional Sobolev regularity under monotone passive feedbacks. To the best of our knowledge, such a preservation result has not previously been established in the framework of regular linear systems.

The paper is organized as follows. Section~\ref{sec:passive} recalls the notions of impedance passive, well-posed and regular linear systems. Section~\ref{sec:stab} introduces the stability notions used throughout the paper. Section~\ref{sec:well-posedness} establishes the well-posedness of the nonlinear feedback interconnection using recent results from
\cite{marx2025impedance,hastir2025well}. Section~\ref{sec:GAS} contains the main contribution of the paper: the propagation of fractional Sobolev regularity through nonlinear passive feedbacks together with its consequences for the asymptotic stability and semi-uniform ISS properties of the closed-loop system. Finally, Section~\ref{sec:conclusion} discusses several perspectives, including extensions to projected passive systems on closed convex sets and the identification of interpolation spaces associated with nonlinear feedback systems.

\section{Impedance passive systems: systems nodes, well-posedness and regular systems}

\label{sec:passive}

We consider three real Hilbert spaces  $H,U,Y$ \footnote{Many results in this article can be extended to the complex case, especially the linear setting.}, where $H$ is the state space, $U$ the input space and $Y$ the output space. In the sequel, given any Hilbert space $W$, we denote by $\mathrm I$ the identity operator of this space, and by $\mathcal L(W_1,W_2)$ the set of bounded (or continuous) linear operators from $W_1$ to $W_2$, with $W_1,W_2$ two real Hilbert spaces. For any operator $A$ (possibly unbounded), we call $\rho(A)$ its resolvent set. We also define $\mathfrak U$ (resp. $\mathscr Y$) the set of functions $f$ satisfying $f\in L^2_{\mathrm{loc}}([0,\infty);U)$ (resp. $f\in L^2_{\mathrm{loc}}([0,\infty);Y)$). We call $\mathbf P_t$ the truncation operator applied to any signals $u\in L^2([0,\infty];U)$. It is defined as follows:
\begin{equation*}
\mathbf P_t u= \left\{
\begin{aligned}
&u(s),\: &\forall &s\in [0,t],\\
&0, &\forall &s >t.
\end{aligned}
\right.
\end{equation*}

As already said in the introduction, this paper is devoted to stritly output passive and impedance passive systems, which are subclasses of systems that appear in many applications.

\begin{definition}[Strictly output passive and impedance passivity]
\label{def:passive}
    Let $H,U,Y$ be three real Hilbert spaces. Suppose $U=Y$. \textit{Strictly output passive systems} satisfy the following energy inequality:
\begin{equation}
\Vert z(t)\Vert^2_H - \Vert z_0\Vert^2_H \leq 2\int_0^t \langle u(s),y(s)\rangle_U \diff s - \mu \int_0^t \Vert y(s)\Vert^2_U \diff s,
\end{equation}
where $\mu\geq 0$, and $z$ denotes the state, $u$ denotes the control and $y$ denotes the output. \textit{Impedance passive systems} satisfy the same energy inequality with $\mu=0$.
\end{definition}

A very common impedance passive system is composed by a dissipative operator $A$, a collocated output (i.e., $C=B^*$) and a dissipative feedtrough operator $D$. More details on such systems are given in \cite{curtain2019strong,curtain2006exponential,slemrod1989mcss}. It can also be proved easily that an impedance passive system with a (unitary) static output feedback is strictly output passive\footnote{It turns out that, for some choices of $\mu>0$, the impedance passive system becomes scattering passive, which is related to the external Cayley transform \cite{curtain2019strong,staffans2002passive}.}.

As illustrated in \cite{staffans2002passive}, impedance passive systems do not need to be well-posed. We will provide examples later on. Hence, we need another concept to deal with general situations. Therefore, we recall the definition of a system node, which is a useful notion for systems that are not well-posed. We borrow this notion from \cite[Definition 4.7.2]{staffans2005well}, but to recall it, we need to define additional spaces and related concepts. Set $\mathbb T$ as the semigroup generated by an operator $A:D(A)\subset H\rightarrow H$ that is densely defined. We call $H_{-1}$ the completion of the space $H$ with respect to the norm $\Vert (s\mathrm I-A)^{-1} z\Vert_H$ with $z\in H$ and $s\in\rho(A)$, and $H_1$ the space $D(A)$ equipped with the norm $\Vert (s\mathrm I-A) z\Vert_H$ with $z\in D(A)$ and $s\in \rho(A)$. The choice of $s$ in the latter definitions of $H_{-1}$ and $H_1$ is not important because the norms are equivalent whatever $s\in\rho(A)$ is. The operator $A$ (resp. $\mathbb T$) can be extended as an operator from $H$ to $H_{-1}$ (resp. as an operator from $H$ to $H_{-1}$), and its extension is also called $A$ (resp. $\mathbb T$) to ease the reading.

\begin{definition}[System nodes]
\label{def:systemnodes}
Let $H$, $U$, $Y$ be three (real) Hilbert spaces. A closed operator:
\begin{equation}
S:=\begin{bmatrix}
A \& B\\
C \& D
\end{bmatrix}:D(S)\subset H\times U\rightarrow H\times Y.
\end{equation}
is called a \textit{system node} on the spaces $(H,U,Y)$ if it has the following properties:
\begin{itemize}
    \item[1.] The operator $A: D(A)\subset H\rightarrow H$ defined by $A\&B\left[
    z\:\: 0
    \right]^\top$ for $z\in D(A):=\lbrace z\in H\mid \left[z \: \: 0\right]^\top \in D(S)\rbrace$ generates a strongly continuous semigroup $\mathbb T$ on $H$.
    \item[2.] The operator $A\& B$ (with domain $D(S)$) can be extended to an operator $\left[
    A \: \: B\right]\in \mathcal L(H\times U,H_{-1})$.
    \item[3.] The domain $D(S)$ is defined by $D(S):=\lbrace (z,u)\in H\times U\mid Az + Bu\in H\rbrace$, with $A,B$ defined in Item (2). 
\end{itemize}
\end{definition}

In other words, if $u=0$ and if there is no output, we have a well-posed system, since the operator $A$ generates a strongly continuous semigroup $\mathbb T$. From the definition of system nodes, one can define the output operator as follows:
$$
Cz = C\& D\left[z \: \: 0\right],\: \forall z\in D(A). 
$$
Once we can define the output operator $C\in\mathcal L(H_1,Y)$ and the input operator $B\in\mathcal L(U,H_{-1})$, we can also define the \textit{transfer function}. To do so, we introduce some notation with respect to operator-valued functions; this notation is borrowed from the nice paper \cite{guiver2017transfer}. Transfer functions belong to this class. For any open set $\Omega\subset \mathbb C$, the set of all holomorphic functions $\Omega \rightarrow \mathcal L(U,Y)$ is denoted by $\mathcal H(\Omega,\mathcal L(U,Y))$. When $\Omega=\mathbb C_\alpha$, where $\mathbb C_\alpha:=\lbrace s\in\mathbb C\mid \mathfrak R_e(s) > \alpha\rbrace$, then we use the following notation 
$$\mathcal H_\alpha(\mathcal L(U,Y)):=\mathcal H(\mathbb C_\alpha, \mathcal L(U,Y)).$$ 

We denote by $\mathcal H^*_\alpha(\mathcal L(U,Y))$ the set of all $\mathcal L(U,Y)$-valued functions which are holomorphic on $\mathbb C_\alpha$, with the exception of some isolated points, typically poles and essential singularities. We denote by $\mathcal H^\infty_\alpha(\mathcal L(U,Y))$ the set of all functions with value $\mathcal L(U,Y)$ and that are bounded. It can be proved that $\mathcal H^\infty_\alpha(\mathcal L(U,Y))\subset \mathcal H_\alpha(\mathcal L(U,Y))$, and the set $\mathcal H^\infty_\alpha(\mathcal L(U,Y))$ is endowed with the following norm:
\begin{equation}
\Vert \mathbf H\Vert_{\mathcal H^\infty(\mathcal L(U,Y))}:=\sup_{s\in \mathbb C_\alpha} \Vert \mathbf H(s)\Vert_{\mathcal L(U,Y)}.
\end{equation}
Now, we can define the transfer function of a system node.
\begin{definition}
The transfer function $\mathbf H$ of a system node is defined as
$$
\mathbf H(s) - \mathbf H(\beta) = C\left[(s\mathrm I_H-A)^{-1}-(\beta \mathrm I_H-A)^{-1}\right]B,
$$
for any $s,\beta\in\mathbb C_{\omega_0(\mathbb T)}$, where $\omega_0(\mathbb T)$ is the growth bound of the semigroup $\mathbb T$. 
\end{definition}

We are now in position to give a definition of what we mean by well-posed systems.

\begin{definition}[Well-posed linear systems]
A time-invariant and well-posed linear system $\Sigma=(\Sigma_t)_{t\geq 0}$ is given by the set of operators $(\mathbb T_t,\Phi_t,\Psi_t,\mathbb F_t)_{t\geq 0}$ that are defined as follows: 
\begin{itemize}
    \item[1.] The semigroup operator $\mathbb T=(\mathbb T_t)_{t\geq 0}\in \mathcal L(H)$
    \item[2.] The input operator $\Phi_t\in\mathcal L(\mathfrak U,H)$.
    \item[3.] The output operator $\Psi_t \in \mathcal{L}(H,\mathscr Y)$.
    \item[4.] The input-output operator $\mathbb F_t\in \mathcal L(\mathfrak U,\mathscr Y)$. 
\end{itemize}
\end{definition}
Such a system is called well-posed because all the operators describing it are bounded yielding the continuous dependence of the state and the output with respect to the initial state and the input. From these operators, one can define the evolution in time of $z\in C([0,\infty);H)$ (which are usually called \textit{mild solutions}) and $y\in \mathscr Y$ (which are usually called \textit{mild outputs}) by the following identities
\begin{equation}
\begin{bmatrix}
z(t)\\\mathbf P_t y
\end{bmatrix} = \begin{bmatrix}
\mathbb T_t & \Phi_t\\
\Psi_t & \mathbb F_t
\end{bmatrix}\begin{bmatrix}
z_0 \\ \mathbf P_t u
\end{bmatrix}.
\end{equation}

More details on well-posed systems are provided in \cite{staffans2005well,tucsnak2014well}. It is now worth wondering how to relate well-posed systems and system nodes. We need for this the notion of admissible input (resp. output) operators, related to the semigroup $\mathbb T$ and its generator $A$. To do so, we specify the operators $\Phi_t$ and $\Psi_t$ as follows:
\begin{equation}
\label{eq:phi}
    \Phi_tu:=\int_0^{t} \mathbb T_{t-s} Bu(s)\diff s,\: \:\Psi_t z_0 :=C\mathbb T_t z_0.
\end{equation}

\begin{definition}
Given $T>0$, we say that $B$ is an admissible input operator in time $T$ for $\mathbb T$ if there exists a positive constant $K_B$ such that the following property is satisfied for every $u\in \mathscr U$ and every $t\in [0,T]$:
$$
\Vert \Phi_t u\Vert^2_{H} \leq K_B^2 \Vert u\Vert^2_{\mathscr U}.
$$
\end{definition}
This means in particular that, for any $t\in [0,T]$, $\mathrm{Ran}\: \Phi_t = H$, i.e. the integration expressed in \eqref{eq:phi} is done in $H_{-1}$, but the result is in $H$. It is worth mentioning that the constant $K_B$ shall depend on $T$.

\begin{definition}
Given $T>0$, we say that $C$ is an admissible output  operator in time $T$ for $\mathbb T$ if there exists a positive constant $K_C$ such that the following property is satisfied for every $z_0\in H$ and every $t\in [0,T]$:
$$
\int_0^t \Vert \Psi_s z_0\Vert_Y^2 \diff s\leq K_C^2\Vert z_0\Vert_H^2.
$$
\end{definition}
Again, here, the constant $K_C$ shall depend on the time $T$. With these notions in mind, we are able to state the following characterization of a well-posed system. 
\begin{definition}
\label{ass:structure}
Let $A$ be the generator of a semigroup $\mathbb T$, $B\in\mathcal L(U,H_{-1})$ and $C\in\mathcal L(H_1,Y)$. The triple $(A,B,C)$ is well-posed if and only if the following properties are satisfied:
\begin{itemize}
    \item[1.] The operator $B$ is admissible for the semigroup $(\mathbb T_t)_{t\geq 0}$. 
    \item[2.] The operator $C$ is admissible for the semigroup $(\mathbb T_t)_{t\geq 0}$.
    \item[3.] The function $s\mapsto \Vert \mathbf H(s)\Vert_{\mathcal{L}(U,Y)}$ is bounded in some right half-plane.\footnote{In the "system theory" community, we say that the transfer function $\mathbf H$ is proper.}
\end{itemize}
\end{definition}
We need to spend few words about the third condition. Indeed, the transfer function is related to the operator $\mathbb F_t$. To obtain the transfer function $\mathbf H$, at least in a very formal way, one only needs to apply the Laplace transform to the operator $\mathbb F_t$. In other words, the third condition implies that the operator $\mathbb F_t$ is bounded on a certain time interval.

Another characterization of well-posed systems is given as follows:

\begin{definition}[Well-posed LTI systems]
\label{def:well-posed}
A system $\Sigma$ defined is said to be well-posed if for some $t>0$ there exists a positive constant $M_t$ such that
$$
\Vert z(t)\Vert_H + \Vert y\Vert_{L^2([0,t];Y)} \leq M_t (\Vert z_0\Vert_H + \Vert u\Vert_{L^2([0,t];U)}).
$$
\end{definition}

Going back to impedance passive systems, let us consider an impedance passive system which is only a system node. Consider now the feedback $u=-\kappa y + d$ with $d\in\mathfrak U$. The energy inequality from Definition \ref{def:passive} becomes

\begin{equation}
\Vert z(t)\Vert^2_H - \Vert z_0\Vert^2_H\leq 2\int_0^t \langle d(s),y(s)\rangle_U \diff s - \kappa \Vert y\Vert^2_{L^2([0,t];U)}
\end{equation}
Using the trivial inequality $2ab\leq \frac{a^2}{\varepsilon}+\varepsilon b^2$ with $\varepsilon=\frac{\kappa}{2}$, one obtains
\begin{equation}
\Vert z(t)\Vert^2_H + \frac{\kappa}{2} \Vert y\Vert^2_{L^2([0,t];U)}\leq  \Vert z_0\Vert^2_H + \frac{1}{2\kappa}\Vert d\Vert_{L^2([0,t];U)},
\end{equation}
which means that the impedance passive system in feedback with the output is well-posed. This will be crucial for our analysis later on.

\begin{example}
\label{example:wave}
We consider a bounded domain $\Omega\subset \mathbb R^n$ with Lipschitz boundary $\Gamma$. $\Gamma_0$ and $\Gamma_1$ are nonempty open subsets satisfying the following: $\Gamma_1 \cap \Gamma_2 = \emptyset$ and $\overline{\Gamma_1 \cup \Gamma_2} = \Gamma$. We consider furthermore a function $b\in L^\infty(\Gamma_1)$ (real valued) satisfying $b(z)\neq 0$ for almost every $z\in \Gamma_1$. The related damped wave equation reads as follows:

\begin{equation}
\label{eq:wave}
    \left\{
\begin{aligned}
    &w_{tt}(t,x) = \Delta w(t,x), &\text{ on } \Omega&\times [0,\infty),\\
    & w(t,x) = 0, &\text{ on } \Gamma_0&\times [0,\infty),\\
    & \frac{\partial}{\partial \nu} w(t,x) = - b(x)y(t,x) + b(x) u(t,x), \quad &\text{ on } \Gamma_1&\times [0,\infty),\\
    & b(x)y(t,x) = b(x)^2 w_t(t,x),  &\text{ on } \Gamma_1&\times [0,\infty), \\
    & w(0,x) = w_0(x),\: w_t(0,x)=w_1(x), \: &\text{ on } \Omega, &
\end{aligned}
    \right.
\end{equation}
where $u$ denotes the input function and $y$ the output function. The functions $w_0$ and $w_1$ are the initial states of the wave equation. 

Now we verify whether this system fits into the strictly output passive setting. To do so, let $W=L^2(\Omega)$ and $U=Y=L^2(\Gamma_1)$, that are Hilbert spaces equipped with usual norms. We denote by $\gamma$ the \textit{Dirichlet trace operator}, defined as follows:
$$
\gamma g = g|_{\Gamma},\: \forall g\in C(\overline{\Omega}).
$$

This operator can be extended from $H^1(\Omega)$ to $L^2(\Gamma)$. We denote by $\mathcal{R}$ the restriction operator mapping $L^2(\Gamma)$ onto $L^2(\Gamma_1)$. One therefore obtains:
$$
\gamma_0 g = \mathcal{R} \gamma g,\quad \forall g\in \Hscr^1(\Omega). 
$$
Then, thanks to this notation, we introduce a Hilbert space $\Hscr^1_{\Gamma_0}(\Omega)$ as follows:
$$
\Hscr^1_{\Gamma_0}(\Omega):=\lbrace g\in H^1(\Omega)\mid (\mathrm I-\mathcal{R})\gamma g = 0\rbrace,\: \Vert g\Vert_{\Hscr^1_{\Gamma_0}(\Omega)}=\Vert \nabla g\Vert_{(L^2(\Omega))^n}.
$$
The space $\Hscr^1_{\Gamma_0}(\Omega)$ is the space of all those functions in $\Hscr^1(\Omega)$ vanishing on $\Gamma_0$. The \textit{Neumann trace operator} on $\Gamma_1$ is as follows:
$$
\gamma_1 g = \frac{\partial}{\partial \nu} f|_{\Gamma_1} = \langle \nabla f,\nu\rangle,\: \forall f\in C^1(\overline{\Omega}),
$$
where $\nu$ is the unit vector in the outward normal direction to $\Gamma$. The operator $\gamma_1$ can be extended to those $f\in \Hscr^1_{\Gamma_0}(\Omega)$ for which $\Delta f\in L^2(\Omega)$, and then $\gamma f$ belongs to a certain Sobolev space on $\Gamma_1$ which includes $L^2(\Gamma_1)$ densely.
 We finally introduce the space
$$W_0:=\lbrace f\in \Hscr^1_{\Gamma_0}(\Omega)\mid \Delta f\in L^2(\Omega),\: \gamma_1f\in bL^2(\Gamma_1)\rbrace,$$
which can be equipped with the following norm:

$$
\Vert w\Vert^2_{W_0}:=\Vert\Delta w\Vert^2_{L^2(\Omega)} + \frac{1}{2}\left\Vert \gamma_1 w\right\Vert^2_{L^2(\Gamma_1)} 
$$
We can rewrite \eqref{eq:wave} as an abstract system (described locally in time). Such systems are called \textit{second order systems}, since they involve a second time derivative. This second order system can be written as follows
\begin{equation}
    \frac{\diff^{\:2}}{\diff t^2} w = A_0 w + \frac{1}{2} B_0 (u(t)- y(t)),\quad y(t) = C_0 \frac{\diff}{\diff t}w(t),
\end{equation}
where $A_0: D(A_0)\subset W \rightarrow W$ is a self-adjoint, positive, and boundedly invertible operator defined as $A_0:=\Delta$. The operator $A_0$ being positive, one can build the space $W_{\frac{1}{2}}=D(A_0^{\frac{1}{2}})$ with the inner product $\langle w_1,w_2 \rangle_{\frac{1}{2}}=\langle A_0^{\frac{1}{2}}w_1,A_0^{\frac{1}{2}}w_2\rangle_W$ and the corresponding norm $\Vert \cdot \Vert_{\frac{1}{2}}$. We denote by $W_{-\frac{1}{2}}$ the dual of $W_{\frac{1}{2}}$ with respect to the pivot space $W$. The state of the system is given by $z(t) = \begin{bmatrix}w(t) & w_t(t)\end{bmatrix}^\top$ and evolves in the state space $H=W_{\frac{1}{2}}\times W$.

We also have
$$
D(A_0) = \lbrace z\in Z_0\mid \gamma_1 z = 0\rbrace, 
$$
and $Z_{\frac{1}{2}}:=H^1_{\Gamma_0}(\Omega)$, meaning that $H=H^1_{\Gamma_0}(\Omega)\times L^2(\Omega)$, equipped with the usual norm:
$$
\Vert (w, w_t)\Vert^2_{H}=\Vert \nabla w\Vert^2_{L^2(\Omega)} + \Vert w_t\Vert_{L^2(\Omega)}.
$$
In order to define $B_0$ and $C_0$, let us introduce the Neumann map $\mathcal{N}\in \mathcal{L}(U,Z_{\frac{1}{2}})$ which satisfies: $\mathcal{N}u=g$ if and only if $g\in H^1_{\Gamma_0}(\Omega)$, $\Delta g =0$ and $\gamma_1g = u$. It is shown in \cite[Section 7]{weiss2003get} that such an operator indeed exists and that, moreover, $\mathcal{N}^*A_0=\gamma_0$. Moreover, one has $\gamma_1\mathcal{N}=\mathrm I$. With these operators, one can define $B_0$ and $C_0$, i.e.
$$
C_0= b \mathcal{N}^*A_0=b\gamma_0,\qquad B_0=C_0^*=A_0 \mathcal{N}b.
$$
Using the fact that $A_0$ is self-adjoint, one can prove that, for any $z_{0}\in X$ and any $u\in L^2_{\mathrm{loc}}([0,\infty);U)$, the generalized solutions to \eqref{eq:wave} satisfy the following: 
\begin{align*}
\frac{1}{2}\Vert z(t)\Vert^2_H \leq & \frac{1}{2} \Vert z_{0}\Vert_H^2+ \langle \mathbf P_t u,\mathbf P_t y\rangle_{\mathscr U} - \mu \Vert \mathbf P_t y\Vert^2_{\mathscr U},
\end{align*}
 with $\frac{1}{\mu}:=\sup_{z\in L^2(\Gamma_1)} b(z)$, which proves that the system \eqref{eq:wave} is strictly output passive.
\end{example}

%\begin{comment}

As explained in \cite[Section 5]{tucsnak2014well}, the concept of well-posed systems might be not satisfactory: related notions of solution are not convenient, and the representation of the output is not easy. Another important notion is the concept of regular linear system. We will introduce it, but, for this, we first need the notion of extended output operators. Indeed, as explained in \cite{tucsnak2009observation,weiss1994regular,weiss1989admissibility}, since $C\in \mathcal{L}(H_1,Y)$ is an admissible output operator for $\mathbb T$, there exists an extension of $C$ defined as:

\begin{equation}
\label{eq:extensionC}
    C_\Lambda z:= \lim_{\lambda\rightarrow + \infty} C\lambda (\lambda \mathrm I_H-A)^{-1}z,
\end{equation}
with $\lambda$ a real number and $z\in D(C_\Lambda)$, where the space $D(C_\Lambda)$ is defined as follows
\begin{equation}
    D(C_\Lambda):=\lbrace z\in H\mid \text{the limit \eqref{eq:extensionC} exists}\rbrace.
\end{equation}
Note that it is not the only extension of $C$, but it turns out that this is sufficient for our purpose. For other extensions, we refer the interested reader to \cite{weiss1994transfer}. 

\begin{definition}
A well-posed triple $(A,B,C)$ is called regular if, for any $u\in U$, one has:
$$
\lim_{\lambda\rightarrow + \infty} \mathbf H(\lambda) u= Du,\: \lambda\in\mathbb R,
$$
where $D\in \mathcal L(U,Y)$. Then, the transfer function of the related system is given by
$$
\mathbf H(s):= C_\Lambda(s\mathrm I_H-A)^{-1} B + D ,
$$
for $s\in \mathbb C$ in the open right half-plane determined by the growth bound of $\mathbb T$, and the output $y$ is defined as
$$
y(t) = C_\Lambda z(t) + Du(t),
$$
for almost every $t\geq 0$.

Moreover, the operator $\mathbb{F}_t$ earlier defined can be defined as follows:
\begin{equation}
\label{eq:F}
\mathbb F_t u:=C_\Lambda \int_0^t \mathbb T_{t-s} B u(s) \diff s + Du(t), 
\end{equation}
for $u\in L^2([0,t];U)$. This operator is defined for some $t$. For $u\in L^2_{\mathrm loc}([0,\infty);U)$, we can introduce the following operator
\begin{equation}
\label{eq:Finf}
\mathbb F_\infty u = C_\Lambda \int_0^t \mathbb T_{t-s} Bu(s) \diff s + Du(t), 
\end{equation}
which is defined for almost every $t\geq 0$.
\end{definition}

%\end{comment}

\section{Stability and observability notions}
\label{sec:stab}
Now, we turn our attention to the notion of exact observability, exponential stability, and input-to-state stability. We will also discuss some properties that can be deduced from exponential stability. We start with the exact observability.

\begin{definition}
Given $T>0$, we say that $C$ is an exact observable output operator in time $T$ for the semigroup $\mathbb T$ if for every initial state $z_0\in D(A)$, there exists $k_C>0$ (depending on time $T$) such that
$$
\int_0^T \Vert \Psi_s z_0\Vert^2_Y\diff s\geq k_C^2 \Vert z_0\Vert^2_H,
$$
\end{definition}
This property, which is the strongest property possible in terms of observability, will be crucial for the construction of the coercive ISS-Lyapunov functional. 

\begin{definition}
We say that $A$ generates an exponentially stable semigroup $\mathbb T$ if there exists $M\geq 1$ and $\gamma>0$ such that, for all $t\geq 0$
$$
\Vert \mathbb T_t\Vert_{\mathcal L(H)} \leq Me^{-\gamma t}
$$
\end{definition}

We call such a stability a uniform stability, in contrast with semi-uniform stability such as polynomial stability \cite{chill2023nonuniform,borichev2010optimal}. If $A$ generates an exponentially stable semigroup $\mathbb T$, then admissible input operators (resp. output operators) in time $T$ are admissible in infinite-time as stated in \cite[Proposition 4.4.5]{tucsnak2009observation} (resp., \cite[Proposition 4.3.3]{tucsnak2009observation}). These notions are defined as follows:

\begin{definition}
We say that $B\in \mathcal L(U,H_{-1})$ is an admissible input operator in infinite-time if there exists a constant $K_B$, independent on the time, such that, for all $t\geq 0$ and all $u\in\mathscr U$,
$$
\Vert \Phi_t u\Vert^2_{H} \leq K_B^2 \Vert \mathbf P_t u\Vert_{\mathscr U}^2.
$$
\end{definition}

\begin{definition}
We say that $C\in \mathcal L(H_{1},Y)$ is an admissible output operator in infinite time if there exists a positive constant $K_C$, independent of the time, such that, for all $t\geq 0$,
$$
\int_0^t \Vert \Psi_s z_0\Vert_Y^2 \diff s \leq K_C^2 \Vert z_0\Vert_{H}^2.
$$
\end{definition}

It turns out that there also exists a notion of exact observability in infinite time, which is given as follows:

\begin{definition}
 We say that $C\in\mathcal L(H_1,Y)$ is exactly observable in infinite-time if there exists a positive constant $k_C$ such that:
$$
\int_0^\infty \Vert \Psi_s z_0\Vert_Y^2 \diff s \geq k_C^2 \Vert z_0\Vert^2_H.
$$
\end{definition}

Now, we focus on the ISS property, which ensures a certain robustness property for the system under consideration.
\begin{definition}
Consider a system defined as follows
\begin{equation*}
\left\{
\begin{aligned}
&\frac{\diff }{\diff t} z(t) = Az(t) + Bu(t),\\
& z(0) = z_0,
\end{aligned}
\right.
\end{equation*}
with $A$ the generator of a semigroup $\mathbb T$ and $B\in \mathcal L(U,H_{-1})$. Is is called exponentially ISS if there exist $M_1,M_2\geq 1$ and $\gamma >0$ such that the mild solution to this system satisfies, for all $t\geq 0$:
$$
\Vert z(t)\Vert_H \leq M_1e^{-\gamma t} \Vert z_0\Vert_H + M_2 \int_0^t e^{-\gamma(t-s)} \Vert u(s)\Vert_{U}\diff s,\: \forall z_0\in H,\: \forall u\in \mathscr U.
$$
\begin{comment}
Moreover, we say that $V:H\rightarrow \mathbb R$, given by $V:=\langle Pz,z\rangle_H$, with $P\in\mathcal L(H)$ self-adjoint and positive, is an ISS Lyapunov functional if the trajectory of \eqref{eq:plant} satisfies, for all $z_0\in Z$ and all $u\in U$
\begin{equation}
\label{eq:ISS}
V(z(t)) \leq Me^{-\lambda t} V(z_0) + \alpha \int_0^t e^{-\lambda(t-s)}\Vert u(s)\Vert_{U}^2\diff s,\quad M,\lambda,\alpha>0.
\end{equation}
\end{comment}
\end{definition}

It turns out that this ISS-Lyapunov functional $V$ does not need to be coercive in an infinite-dimensional setting, as illustrated in \cite{hante2011converse,Mironchenko2018Characterizations,jacob2020noncoercive}. However, recently, it has been proved in \cite{marx2025coercive} that, as soon as one is able to find up an infinite-time exactly observable output operator, then it is possible to build a coercive Lyapunov functional. It relies, moreover, on the notion of regular linear system developed in \cite{weiss1994transfer}.

 \section{Well-posedness of an impedance passive system in feedback with a saturation}

 \label{sec:well-posedness}

Now, we consider an impedance passive system $\Sigma$ as in Definition \ref{def:passive} in feedback with a saturation $\sigma$ as illustrated in Figure \ref{fig:sat}. Associated to the notion of saturation, and following what has been proposed in \cite{marx2018stability} for the case of an infinite-dimensional linear system with a saturated collocated feedback, we introduce the notion of Gelfand-Banach triple.
\begin{definition}[Gelfand-Banach triple]
Given $U$ a real Hilbert space, we define a Gelfand-Banach triple as the triple $\mathcal S^*\subset U\subset \mathcal S$ with $\mathcal S^*$ the topological dual of $\mathcal S$ and where $\mathcal S^*$ (resp. $U$) is continuous and dense in $U$ (resp. $\mathcal S$). For every function $(u_1,u_2)\in S^*\times U$, we can identify the scalar product of $U$ as follows
\begin{equation}
\langle u_1,u_2\rangle_U:=(u_1,u_2)_{\mathcal S^*,\mathcal S},
\end{equation}
where $(\cdot,\cdot)_{\mathcal S^*, \mathcal S}$ is the dual product of the spaces $\mathcal S^*$ and $\mathcal S$. 
\end{definition}
A well known Gelfand-Banach triple is given by $U=L^2(\Omega)$ (with $\Omega\subset \mathbb R^n$ bounded), $\mathcal S:=L^1(\Omega)$ and $\mathcal S^*=L^\infty(\Omega)$. Obviously, if $U=\mathbb R^n$, then $\mathcal S=\mathcal S^*=\mathbb R^n$. Note that every function $u\in U$ belongs to $\mathcal S$. As illustrated in \cite{marx2018stability}, such a formalism allows to consider more general types of saturations. We define them as follows. 

\begin{figure}[h!]
\label{fig:sat}
\begin{center}
\includegraphics[scale=0.3]{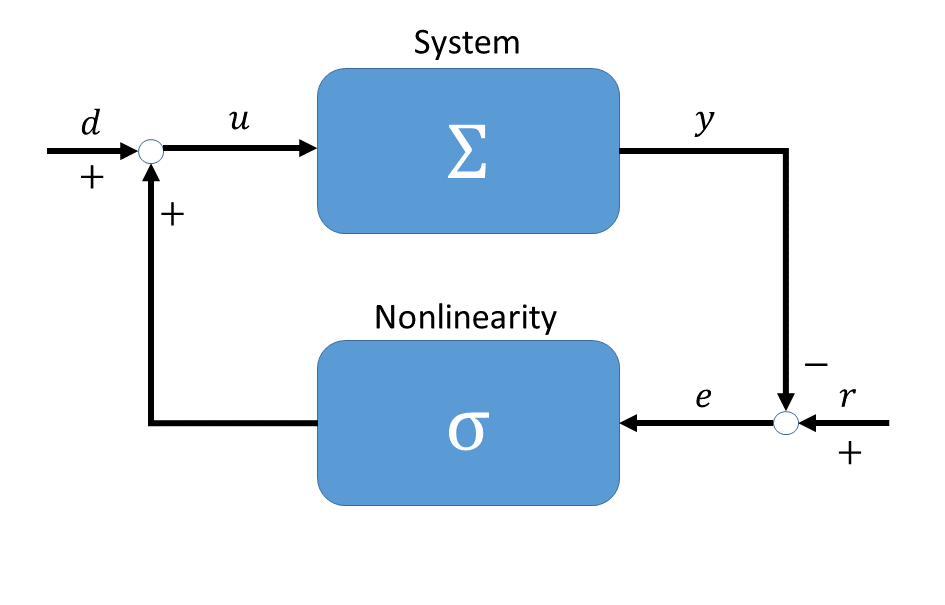}
\caption{Regular linear system in feedback with a saturation $\sigma$.}
\end{center}
\end{figure}

\begin{definition}
\label{def:sat}
Let $(\mathcal S^*,U,\mathcal S)$ be a Gelfand-Banach triple. The function $\sigma: U\rightarrow U$ is called a saturation if it satisfies the following properties:
\begin{itemize}
    \item[(i)] The function $\sigma$ is strictly output passive, i.e., for every $y_1,y_2\in U$
    \begin{equation}
    \label{eq:strict-output}
        \langle\sigma(y_1)-\sigma(y_2),y_1-y_2\rangle_U \geq \Vert \sigma(y_1)-\sigma(y_2)\Vert^2_U.
    \end{equation}
    \item[(ii)] There exist positive constants $\delta,\beta,\gamma>0$ such that
    \begin{equation}
    \left\{
    \begin{aligned}
    &\langle y,\sigma(y)\rangle_U \geq \beta \Vert y\Vert^2_U,&\:\text{ if }& \Vert y\Vert_{\mathcal S^*}\leq \delta,\\
    &\langle y,\sigma(y)\rangle_U \geq \gamma\Vert y\Vert_{\mathcal S}, &\text{ if }& \Vert y\Vert_{\mathcal S^*}\geq \delta,
    \end{aligned}
    \right.
    \end{equation}
    \item[(iii)] There exists a positive constant $K_\sigma$ such that, for every $y_1,y_2\in U$, one has
    \begin{equation}
    \label{eq:ineq-sat-ISS}
        \langle y_1,\sigma(y_1+y_2)-\sigma(y_1)\rangle_U \leq K_\sigma \Vert y_2\Vert_{\mathcal S}.
    \end{equation}
    \item[(iv)] There exists a positive constant $C_\sigma$ such that, for all $y\in \mathcal S^*$
    \begin{equation}
    \label{eq:bounded-sat}
    \Vert \sigma(y)\Vert_{\mathcal S^*}\leq C_\sigma
    \end{equation}
    \item[(v)] For any $s\in U$ (and therefore $s\in \mathcal S$), we have that
    \begin{equation}
    \label{eq:key-inequality}
    \Vert \sigma(s)-s\Vert_{\mathcal S}\leq \langle \sigma(s),s\rangle_U.
    \end{equation}
    \end{itemize}
    
\end{definition}

In the following, we will need to consider $\sigma$ depending on signals $y\in L^2_{\mathrm{loc}}([0,\infty);Y)$. For any $\tau>0$, we define an operator $\tilde\sigma$ from
$L^2([0,\tau];Y)$ to $L^2([0,\tau];U)$ as the pointwise application of
$\sigma$, as follows:
$$ v= \tilde{\sigma}(y)\quad\ \text{if}\ \quad v(t)= \sigma(y(t)) 
\quad \text{for almost every}\quad t\in [0,\tau].$$
In order to avoid any confusion, we will use the notation $\tilde{\sigma}:=\sigma$. We say that such a nonlinearity is \textit{local} (in contrast with non-local nonlinearities). This local property ensures that the closed-loop system remains causal.

Definition \ref{def:sat} is a modification of the definition of saturation given in \cite{curtain2016stabilization,guiver2020circle} and also used in \cite{slemrod1989mcss}, or more recently in \cite{hastir2025well}. The use of the Gelfand-Banach triple $(\mathcal S^*,U,\mathcal S)$ - a concept which has already been introduced in \cite{marx2018stability,jacob2020remarks} - allows to consider a larger class of saturations. We illustrate this fact with several examples.

\begin{example}[Saturations]
Suppose that $U=\mathbb R^2$. Therefore $\mathcal S=\mathcal S^* = U$. The classical saturation $\sat:\mathbb R^2\rightarrow \mathbb R^2$ is defined by
\begin{equation}
\label{eq:sat-usual}
\sat_i(s_1):=\left\{
\begin{aligned}
&s &\text{ if }& |s_1|\leq 1,\\
&1 &\text{ if }& s_1 \geq 1,\\
&-1 &\text{ if }& s_1 \leq - 1,
\end{aligned}
\right.
\end{equation}
where $i=\lbrace 1,2\rbrace$ and $\sat_i$ denotes the $i$-th component of the vector $\sat(s)$. One can easily prove that this function satisfies the properties in Definition \ref{def:sat}.

Suppose now that $\Omega\subset \mathbb R^2$ is bounded and consider the space $U=L^2(\Omega)$. Since $\Omega$ is bounded, one can define the Gelfand-Banach triple $(\mathcal S^*,U,\mathcal S)$ with $\mathcal S:=L^1(\Omega)$ and $\mathcal S^*:=L^\infty(\Omega)$. Consider $y\in L^2(\Omega)$. One can easily prove that the inequalities given in Definition \ref{def:sat} are naturally satisfied for the function $\sat\circ y:U\rightarrow U$.

We can also define the following saturation:
\begin{equation}
\label{eq:sat-U}
\mathrm{sat}_U(u):=\left\{\begin{aligned}
&\hspace{0.38cm}u & \text{ if } & \Vert u\Vert_U \leq 1\\
&\frac{u}{\Vert u\Vert_U} & \text{ if } & \Vert u\Vert_U \geq 1.
\end{aligned}
\right.
\end{equation}
This nonlinearity, even if it looks similar to \eqref{eq:sat-usual}, is very different, and one cannot expect the same asymptotic results, as illustrated in \cite{mcpa2017siam,marx2018stability}. In general, saturation as \eqref{eq:sat-U} allows one to obtain stronger result, see, e.g., \cite{marx2018stability}. 
\end{example}

For the analysis of trajectories provided in Section \ref{sec:GAS}, we will not treat the case of $\mathcal S\neq U$, but some comments about this case are given in Section \ref{sec:conclusion}.

The well-posedness of nonlinear infinite-dimensional systems is obviously more intricate than in the linear setting. The general tool to tackle this issue relies on nonlinear semigroup theory \cite{barbu2010nonlinear, brezis1973operateurs, miyadera1992nl_sg} and more precisely maximal monotone operators theory, but during decades this theory only considered the case without input and output. The input case has been considered in \cite{Mironchenko2020Input} in the context of ISS for infinite-dimensional systems. Very recently, some results were obtained considering inputs and outputs in \cite{singh2020non,singh2021abstract,singh2023second,singh2024local,marx2025impedance} in the case of passive systems, and more precisely, the case of incrementally scattering passive systems (at least for \cite{singh2024local}), allowing to define a nonlinear extension of the Lax-Phillips semigroup presented for the first time (in its linear form) in \cite{lax1990scattering} and discussed in detail in \cite{staffans2005well}. We should therefore consider such a setting, that is, a passive setting. To do so, the following assumption will be made.

\begin{assumption}
\label{ass:impedance}
Assume $U=Y$. We assume that $\Sigma$ is an impedance passive system (possibly a system node) and we assume that $A$ is maximal dissipative (hence, it generates a strongly continuous semigroup  of contractions called $\mathbb T=(\mathbb T_t)_{t\geq 0}$).
\end{assumption}

Indeed, supposing that Assumption \ref{ass:impedance} holds, we do not know whether the system is well-posed. We have seen in Section \ref{sec:passive} that an impedance passive in feedback with the output is well-posed. In order to apply \cite{marx2025coercive}, we need an additional regularity property which is given as follows:

\begin{assumption}
\label{ass:regular}
Let $\Sigma$ an impedance passive system. We suppose that the system $\Sigma$ in feedback with $u=-\kappa y + d$, with $d\in \mathfrak U$, is regular. 
\end{assumption}

In other words, this feedback system (that is more than well-posed, i.e., it is regular\footnote{We think that the results can be extended to the well-posed case (and more precisely, compatible systems), but for notation convenience, we prefer to keep the regularity assumption.}) can be represented as follows

\begin{equation}
\label{eq:plant}
\left\{
\begin{aligned}
&\frac{\diff }{\diff t} z(t) = Az(t) + Bu(t),\\
&y(t) = C_\Lambda z(t) + Du(t),\\
& u(t) = - \kappa y(t) + d(t),\\
& z(0) = z_0,
\end{aligned}
\right.
\end{equation}
where $C_\Lambda$ is the $\Lambda$-extension of $C$. 

\begin{assumption}
\label{ass:stab}
For any $\kappa>0$, we suppose that the operator $A-\kappa B(\mathrm I+\kappa D)^{-1}C_\Lambda$ generates a strongly continuous semigroup $(\mathbb S_t)_{t\geq 0}$ that is exponentially stable, and we also have the resulting output operator $(\mathrm I + \kappa D)^{-1}C_\Lambda$ which is exactly observable in infinite-time.
\end{assumption}

This assumption is restrictive since it might happen that an infinite-dimensional system can be globally asymptotically stable, but not exponentially stable (which does not occur for finite dimensional systems). Under this assumption (which is similar to the one assumed in \cite{liu1996finite,marx2018stability,liu1997locally,curtain2006exponential}), we will modify the feedback law as follows $u=-\sigma(\kappa e)+d$ with $e=y-r$ (see Figure \ref{fig:sat}). In the linear case, it is equivalent to consider just a $d$ instead of a $d$ and a $r$. This is no longer the case for nonlinear systems.

Before stating our first main result, we need to have a definition of what is a well-posed nonlinear system (at least, in the specific case discussed here) such as the one given in Figure \ref{fig:sat}. To do so, we will consider the following implicit operators:

\begin{equation}
\label{eq:fixed-point}
\left\{
\begin{aligned}
&z(t) &=&\: \mathbb T_t z_0 + \Phi_t \mathbf P_t (-\sigma(\kappa e(t)) + d(t)),\\
&\mathbf P_t y(t) &=&\: \Psi_t z_0 + \mathbb F_t \mathbf P_t (-\sigma(\kappa e(t)) + d(t)),\\
& e(t) &=& \: \mathbf P_t y(t)-\mathbf P_t r(t).
\end{aligned}
\right.
\end{equation}
By implicit, we mean that there is no explicit way to express the state $z$ and the output $y$ with respect to $z_0\in H$ and the input $u\in L^2_{\mathrm{loc}}([0,\infty);U)$. It is worth saying that the operators $\Phi_t$, $\Psi_t$ and $\mathbb F_t$ do not need to be bounded, as we are looking at a system node (see \cite[Section 4.7]{staffans2005well} for a precise definition of these operators). In particular, if $\Phi$ can be linked with the operators $A,B,C,D$ it is not the case of the operators $\Psi$ and $\mathbb F$. However with Assumption \ref{ass:regular} we claim that we will be able to prove that the system is well-posed. Roughly speaking, the system $\Sigma^{\sigma}$ given in Figure \ref{fig:sat} is expressed with fixed-point functions given in \eqref{eq:fixed-point}. 

\begin{comment}
We are now ready to define what we mean by well-posedness of the system given in Figure \ref{fig:sat}, and that can be written as follows:

\begin{equation}
\label{eq:system-sat}
\left\{
\begin{aligned}
&\frac{\diff}{\diff t} z = Az-B(\sigma(\kappa e(t)) + d(t)),\\
&y(t) =  C_\Lambda z(t) - D(\sigma(\kappa e(t))+d(t)),\\
&e(t) = y(t) - r(t),\\
& z(0) = z_0
\end{aligned}
\right.
\end{equation}
\end{comment}

\begin{definition}
\label{def:mild-solution}
We say that there exists a unique mild solution to the system $\Sigma^\sigma$ given in Figure \ref{fig:sat} if, for any $z_0\in H$ and any $u\in L^2_{\mathrm{loc}}([0,\infty);U)$, there exist a unique $z\in C([0,\infty);H)$ and a unique $y\in L^2_{\mathrm{loc}}([0,\infty);Y)$ satisfying \eqref{eq:fixed-point}, and if, moreover, $z$ and $y$ continuously depend on $z_0$ and $u$.
\end{definition}

It is worth noticing that \eqref{eq:fixed-point} can be equivalently written as:

\begin{equation}
\label{eq:fixed-point1}
\left\{
\begin{aligned}
&z(t) &=&\: \mathbb T_t z_0 - \Phi_t(\mathrm I + \kappa \mathbb F_t)^{-1} \kappa \Psi_t z_0 + \Phi_t (\mathrm I + \kappa \mathbb F_t)^{-1} \mathbf P_t v,\\
&\mathbf P_t y &=&\: (\mathrm I + \kappa \mathbb F_t)^{-1}\Psi_t z_0 + \mathbb F_t (\mathrm I + \kappa \mathbb F_t)^{-1} \mathbf P_t v,\\
& \mathbf P_t e &=& \: \mathbf P_t(y-r),
\end{aligned}
\right.
\end{equation}
by rewriting the system \eqref{eq:plant} (given with $u=-\sigma(\kappa e)+d$) as
\begin{equation}
\label{eq:cl-sat}
\Sigma_B:\left\{
\begin{aligned}
&\frac{\diff}{\diff t} z(t) = (A-\kappa B (\mathrm I+\kappa D)^{-1}C_\Lambda)z(t) + B(\mathrm I + \kappa D)^{-1}v(t),\\
& y(t) = (\mathrm I + \kappa D)^{-1}C_\Lambda z(t) + (\mathrm I + \kappa D)^{-1}Dv(t),\\
& z(0)=z_0,
\end{aligned}
\right.
\end{equation}
with \begin{equation}
\label{eq:expression-input}
v(t) = \kappa e(t) - \sigma(\kappa e(t)) + d(t). 
\end{equation} 
We can write the operators $A,B,C_\Lambda$ and $D$ because we have assumed the linear system (with $v=0$) to be regular. Since the closed-loop system is supposed to be regular (see Assumption \ref{ass:regular}), we see that the following operators are bounded (in suitable functional spaces):
\begin{equation}
\label{eq:new-operators}
\left\{
\begin{aligned}
&\mathbb S_t:=\mathbb T_t - \Phi_t(\mathrm I + \kappa \mathbb F_t)^{-1} \kappa \Psi_t,\\
&\tilde\Phi_t:=\Phi_t (\mathrm I + \kappa \mathbb F_t)^{-1}\\
&\tilde\Psi_t:= (\mathrm I + \kappa \mathbb F_t)^{-1}\Psi_t\\
&\tilde{\mathbb F}_t:=\mathbb F_t (\mathrm I + \kappa \mathbb F_t)^{-1},
\end{aligned}
\right.
\end{equation}
associated with the generators:
\begin{equation}
\label{eq:new-generators}
\left\{
\begin{aligned}
& A_\kappa:= A-\kappa B (\mathrm I+\kappa D)^{-1}C_\Lambda,\\
& B_\kappa:=B(\mathrm I + \kappa D)^{-1},\\
& C_\kappa:= (\mathrm I + \kappa D)^{-1} C_\Lambda,\\
& D_\kappa:= (\mathrm I + \kappa D)^{-1}D.
\end{aligned}
\right.
\end{equation}

This situation corresponds to Figure \ref{fig:dz}.
\begin{figure}[h!]
\label{fig:dz}
\begin{center}
\includegraphics[scale=0.3]{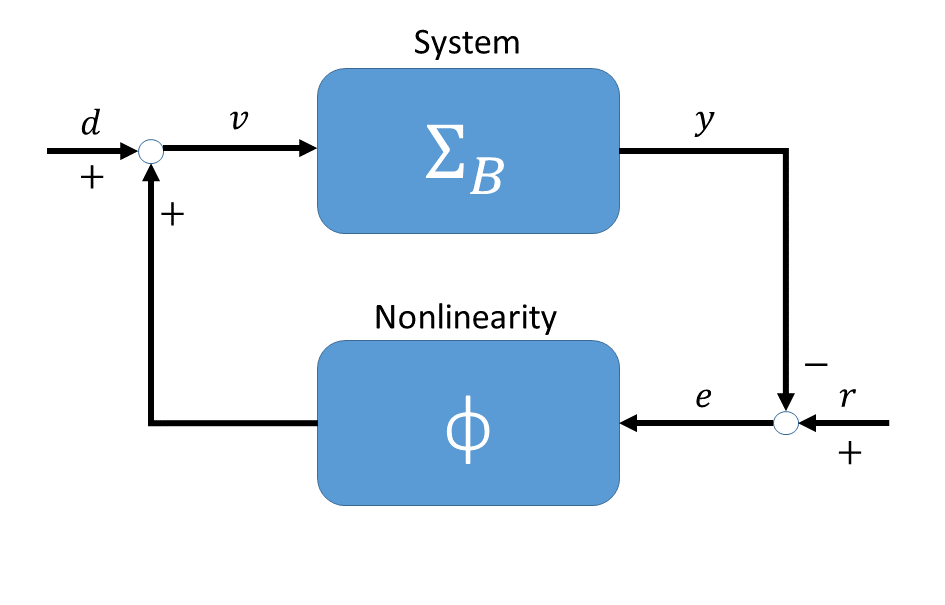}
\caption{Regular linear system in feedback with a dead-zone $\phi$.}
\end{center}
\end{figure}
In Figure \ref{fig:dz}, $\Sigma_B$ given in \eqref{eq:cl-sat} with $v\in\mathfrak U$ an arbitrary input, and the function $\phi(e):= \sigma(\kappa e)-\kappa e$ for every $e\in \mathfrak U$. As said in the caption of Figure \ref{fig:dz}, such functions $\phi$ are called dead-zone functions in the saturation literature \cite{tarbouriech2011book_saturating}. 

Now, we are in position to state our first main result: 

\begin{theorem}
\label{thm:mild}
For any initial state $z_0\in H$ and any inputs $d,r\in L^2_{\mathrm{loc}}([0,\infty);U)$, there exists a unique mild solution to \eqref{eq:cl-sat}. Moreover, there exists $\mu>0$ such that each mild solution $z_1,z_2$ and any output $y_1,y_2$ satisfy the following inequality for any given $z_{01},z_{02}\in H$ and any given $d_1,d_2,r_1,r_2\in L^2_{\mathrm{loc}}([0,\infty);U)$, and for every $t\geq 0$
\begin{equation}
\label{eq:incremental}
\begin{aligned}
\Vert z_1(t)-z_2(t)\Vert^2_{H} - \Vert z_{01}-z_{02}\Vert^2_H \leq & \left\langle \begin{bmatrix}
\mathbf P_t d_1-\mathbf P_t d_2\\ \mathbf P_t r_1-\mathbf P_t r_2
\end{bmatrix},\begin{bmatrix}
\mathbf P_t y_1 
- \mathbf P_ty_2\\ \sigma(\mathbf P_t e_1)-\sigma(\mathbf P_t e_2)
\end{bmatrix}\right\rangle_{\mathfrak U\times \mathfrak U}\\
&-\mu \left\Vert \begin{bmatrix}
\mathbf P_t y_1 - \mathbf P_t y_2\\ \sigma(\mathbf P_t e_1)-\sigma(\mathbf P_t e_2)
\end{bmatrix}\right\Vert^2_{\mathfrak U\times \mathfrak U} 
\end{aligned}
\end{equation}
\end{theorem}

The proof of this theorem is based on the application of a recent result \cite[Theorem 5.5]{marx2025impedance} dealing with impedance passive systems, which states that a well-posed strictly output passive system in feedback with a strictly output passive feedback operator is well-posed and incrementally strictly output passive (which corresponds to the property in \eqref{eq:incremental}). We will not recall this result, nor prove Theorem \ref{thm:mild}, as it is just an application of \cite[Theorem 5.5]{marx2025impedance}. As an assumption, we already know that $\sigma$ is an incrementally strictly output passive operator (see e.g., Assumption \ref{def:sat}) and so is the dead-zone function $\phi$. It remains to show that the open-loop system $\Sigma_B$ is strictly output passive, which is true because of Assumption \ref{ass:impedance}. 

It is worth noting that \cite{hastir2025well} also proposes a well-posedness analysis, but an additional assumption on the transfer function is needed. Namely, the transfer function needs to be coercive for one frequency, which is a fair assumption when one restricts the analysis to one-dimensional PDEs (or, more precisely, when $U=\mathbb R^m$), but which is stronger and even restrictive for more general PDEs. However the result in \cite{hastir2025well} allows to consider strong solutions, and we are not aware whether such an extension holds for the class of systems where the coercivity fails. Nevertheless, we will recall some results obtained in this paper; they will be useful to obtain stronger results than polynomial stability. To do so, we need some notation. We define the nonlinear operator
$$
A_\phi:\: D(A_\phi)\subset H\rightarrow H,
$$
with the domain defined as
$$
D(A_\phi):=\lbrace z\in H\mid \exists v\in U \text{ s.t. } A_\kappa z + B_\kappa \phi(-v)\in H,\: v= C_\kappa z + D_\kappa \phi(-v)\rbrace,
$$
and
$$
A_\phi z:=A_\kappa z + B_\kappa v,\quad v= C_\kappa z + D_\kappa \phi(-\kappa v).
$$
We therefore have the following result:

\begin{theorem}[\cite{hastir2025well}, Theorem 3.9.]
\label{thm:hastir}
Given an arbitrary input $v\in \mathfrak U$, consider \eqref{eq:cl-sat} a strictly output passive system whose transfer function $\mathbf H$ satisfies 
\begin{equation}
\label{eq:coercivity-tf}
\mathfrak R_e \mathbf H(\lambda) \geq c_\lambda \mathrm I
\end{equation}
for some $\lambda,c_\lambda >0$. Consider a nonlinearity $\phi=\mathrm I-\sigma$, where $\sigma$ is a saturation as in Definition \ref{def:sat}. Therefore, for $d=r=0$, and for any $z_0 \in \overline{D(A_\phi)}$, there exists a unique mild solution to \eqref{eq:cl-sat}. Moreover, for $d=r=0$, and for any $z_0 \in D(A_\phi)$, $z$ is a strong solution to \eqref{eq:cl-sat}, that is, $z$ satisfies \eqref{eq:cl-sat} for almost every $t\geq 0$ and we have $z(t)\in D(A_\phi)$. We also have that $y$ is a right continuous function satisfying
\begin{equation}
\label{eq:Linfty}
y\in L^\infty([0,\infty);U).
\end{equation}
\end{theorem}

The proof given in \cite{hastir2025well} is based on maximal monotone operator theory \cite{miyadera1992nl_sg,brezis1973operateurs}, and consists of proving that $A_\phi$ is maximal monotone. Maximality, in particular, requires this additional assumption on the transfer function, which is not needed in \cite{marx2025impedance}. However, in the latter reference, it is not proved that there exists a strong solution.

\begin{remark}[Lack of coercivity]
\label{rem:lack-coercivity}
The assumption given in \eqref{eq:coercivity-tf} is verified in most cases for one-dimensional PDEs, i.e., for $U$ reduced to be finite-dimensional. Indeed, suppose that we have a compact transfer function $\mathbf H\in \mathcal L(U)$ with $U$ an infinite-dimensional space. Then, take an orthonormal sequence $(e_n)_{n\in\mathbb N}\subset U$. By Bessel inequality \cite[Theorem 3.4-6]{analysis2007Kreyszig}, it turns out that $e_n$ converges weakly to $0$, i.e., for every $u\in U$, one has $\langle u,e_n\rangle_U\rightarrow 0$. Since $\mathbf H$ is a compact operator, then $\mathbf H e_n$ converges strongly \cite[Theorem 8.1-7]{analysis2007Kreyszig} to $0$, which implies that \eqref{eq:coercivity-tf} cannot hold.

Going back to the example given in Example \ref{example:wave}, it can be shown that the transfer function of \eqref{eq:wave} is compact. Indeed, it can be written as follows
$$
\mathbf H(s)\hat u(s,x):= b(x) s \gamma_1\hat w ,
$$
where $s$ is the Laplace variable, $\hat u$ (resp. $\hat w$) denotes the Laplace transform of $u$ (resp. $w$) with respect to the time-variable, where we recall that $\gamma_1$ is the Neumann trace operator, and where $\hat w$ is the solution to the following problem:
\begin{equation}
\left\{
\begin{aligned}
&(s^2-\Delta) \hat w = 0,\: &\text{ on }& \Omega,\\
& \hat w = 0,\: &\text{ on }& \Gamma_0,\\
& \frac{\partial}{\partial \nu} \hat w + s b^2\hat w = b \hat u,\: &\text{ on }& \Gamma_1.
\end{aligned}
\right.
\end{equation}
Under sufficient boundary regularity (for instance $\Omega$ of class $C^{1,1}$), by standard elliptic regularity and the trace theorem, we have  $\mathbf H(s) \in\mathcal L(L^2(\Gamma_1),H^{1/2}(\Gamma_1))$, which is sufficient to prove that $\mathbf H$ is a compact operator for all $s\in \rho(A)$. Therefore, Assumption \ref{eq:coercivity-tf} cannot hold. This illustrates that the latter assumption is naturally restricted to systems whose input space is finite-dimensional, such as one-dimensional boundary control systems. 

Indeed, if we consider Example \ref{eq:wave} in its one-dimensional version, the transfer function writes
$$
\mathbf H(s):=\frac{b^2 \tanh(s)}{1+b^2 \tanh(s)},
$$
which can be proved to be positive. Since $U=\mathbb R$ (and therefore, $\mathcal S=U$) in this case, Assumption \ref{eq:coercivity-tf} holds in the 1D case. 
\end{remark}

\section{Stability results}
\label{sec:GAS}

Before presenting our main results, we need to introduce some useful notions for nonlinear systems. 
\begin{definition}
We say that a function $\ell:\mathbb R_+\rightarrow \mathbb R_+$ is a $\mathcal K_\infty$-function if it is strictly increasing, vanishing at $0$ and such that $\lim_{s\rightarrow + \infty} \ell(s)=+\infty$. We say that a function $\rho:\mathbb R_+\times\mathbb R_+\rightarrow \mathbb R_+$ is a $\mathcal{KL}$-function if for all $t\in\mathbb R_+$, $\rho(t,\cdot)$ is a $\mathcal{K}_\infty$-function and if for all $s\in\mathbb R_+$, $\rho(\cdot,s)$ is decreasing and such that $\lim_{t\rightarrow +\infty} \rho(t,s)=0$ (again, for all $s\in\mathbb R_+$).
\end{definition}

These classes of function have been proved to be very useful in the (nonlinear) finite-dimensional setting, as it is nicely illustrated in \cite{bible_khalil,sontag1995characterizations,isidori1985nonlinear}. It is therefore obvious that such functions are also useful in the infinite-dimensional setting as it has been already established in \cite{Mironchenko2020Input,chitour2020one}.

In such a general setting, one cannot hope having a uniform global asymptotic stability when the disturbances $d$ and $r$ are equal to $0$. In particular, a counter-example is given in \cite[Theorem 7]{jacob2020remarks}, where it is shown that a uniform global asymptotic stability does not hold for a transport equation with a bounded input operator and a saturation where $\mathcal S:=L^1([0,1])$. However, as proved in \cite{marx2018stability} (and in a more specific way, \cite{chitour2019p}), a semi-global exponential stability holds for bounded input operators and saturation where $\mathcal S\neq U$, which means (obviously) that semi-global exponential stability does not imply uniform global asymptotic stability. Another example is provided in \cite[Section 4.5]{chitour2020one} for a one-dimensional wave equation with nonlinear boundary damping (hence, an unbounded input operator in open-loop), and, in this case, we have $\mathcal S=U=\mathbb R$. In other words, uniform stability results (ISS or just uniform global asymptotic stability) cannot be obtained when a saturation is involved for most infinite-dimensional systems.

Based on the latter paragraph, we specify in the next definition what we mean by uniform ISS and semi-uniform ISS properties.

\begin{definition}[Concepts of stability]
\label{def:stability-concepts}
The different concepts needed for the sequel are defined as follows. 
\begin{itemize}
\item[(1)] System \eqref{eq:cl-sat} is called \textit{uniformly ISS} if if there exist a $\mathcal{KL}$-functions $\rho$ and a $\mathcal K_\infty$-function $\ell$ such that, for all $t\geq 0$
\begin{equation}
\Vert z(t)\Vert_H \leq \rho(t,\Vert z_0\Vert_H) + \ell(\Vert \mathbf P_t d\Vert^2_{\mathfrak U}, \Vert \mathbf P_t r\Vert^2_{\mathfrak U}).
\end{equation}
It is called \textit{uniformly globally asymptotically stable} (UGAS for short) if $d=r=0$.
\item[(2)] Suppose that $d$ is uniformly bounded, i.e., there exists $K_d>0$ such that $\Vert d(t)\Vert_U\leq K_d$ for almost every $t\geq 0$. Given a dense subspace of $H$, denoted by $H_s$, system \eqref{eq:cl-sat} is called \textit{semi-uniformly ISS} if there exist a $\mathcal{KL}$-functions $\rho$ and a $\mathcal K_\infty$-function $\ell$ such that, for all $t\geq 0$
\begin{equation}
    \Vert z(t)\Vert_H\leq \rho(\Vert z_0\Vert_{H_s},t)+\ell(\Vert \mathbf P_t r\Vert_{\mathfrak R}),
\end{equation}
where $\mathfrak R:=L^1_{\mathrm{loc}}([0,\infty);U).$ The functions $\rho$ and $\ell$ depend on the bound $K_d$.
It is called \textit{semi-uniformly globally asymptotically stable} (SUGAS for short) if $d=r=0$.
\item[3.] Given a dense subspace of $H$, denoted by $H_s$, system \eqref{eq:cl-sat} is called \textit{semi-globally exponentially ISS} if it is semi-uniformly ISS with $\rho(s,t) = \psi (s) e^{-\lambda t}$, where $\psi$ is a continuous function vanishing at $0$, strictly increasing and diverging as $s$ goes to infinity. It is called \textit{semi-globally exponentially stable} (SGES) if it is semi-globally exponentially ISS with $d=r=0$.
\end{itemize}
\end{definition}

\begin{remark}
It is worth noting that the ISS estimate is different from what is known in the literature \cite{Mironchenko2020Input}. It is due to the fact that $d$ needs to be bounded and because the estimate in \eqref{eq:ineq-sat-ISS} imposes to look at $L^1$-estimates.
\end{remark}

Here, as will be shown later on, even if the linear system is exponentially stable (or if $\sigma$ is reduced to be the identity operator), it turns out that in most cases, the asymptotic stability is not uniform. This is mainly due to the high amplitude in the output. Recall that the energy function $\frac{1}{2}\Vert z\Vert_H^2$ satisfies, for all $t\geq 0$ and all $z_0\in H$
$$
\Vert z(t)\Vert_H^2 - \Vert z_0\Vert_H^2 \leq - 2 \int_0^t \langle \sigma(y),y\rangle_U \diff t,
$$
with $u$ and $d$ identically equal to $0$. For the outputs $y$ satisfying $\Vert y(t)\Vert_U\geq \delta$, it turns out that
$$
\Vert z(t)\Vert_H^2 - \Vert z_0\Vert_H^2\leq - 2\gamma \Vert y\Vert_{L^1([0,t];U)}.
$$
Therefore, one is left with the following question: is it possible to deduce $L^1$-exact observability from a $L^2$-exact observability ? This question has been tackled in \cite{zuazua2012remark}: in the case of bounded input operator the answer is positive, but the property has been proved to be false for unbounded input operator with the help of a simple transport equation. 

Before stating and proving our main results, we define the function
\begin{equation}
\mathcal J(r):= \left\{
\begin{aligned}
& r^2 &\text{ if } r\leq \delta,\\
& r &\text{ if } r\geq \delta,
\end{aligned}
\right.
\end{equation}
and we suppose that there exist the positive constants $\overline c_{\sigma},\underline c_{\sigma}>0$ such that
\begin{equation}
\label{eq:jauge}
\underline c_{\sigma} \mathcal J(\Vert y\Vert_U) \leq \langle \sigma(y),y\rangle_U \leq \overline c_{\sigma} \mathcal J(\vert y\Vert_U),
\end{equation}
which is satisfied for any saturation given in Definition \ref{def:sat} when $\mathcal S=U$. In the sequel, an estimate will be given in terms of the function $\mathcal J$. To make the reading clearer, we give some notation for signals $y\in L^p_{\mathrm{loc}}([0,\infty);U)\cap \mathfrak U$. For any $p\in [1,\infty)$, we denote by $\Vert \cdot \Vert_p$ the norm
$$
\Vert y\Vert^p_p:=\int_0^{\infty} \Vert y(t)\Vert_U^p\diff t,
$$
and for $p=\infty$, we have
$$
\Vert y\Vert_{\infty}:=\esssup_{t\in [0,\infty]} \Vert y(t)\Vert_U.
$$
It is clear that, for any $T>0$
$$
\Vert \mathbf P_T y\Vert_p^p = \int_0^T \Vert \mathbf P_T y(t)\Vert_U^p \diff t,\quad \Vert \mathbf P_T y\Vert_{\infty}=\esssup_{t\in [0,T]} \Vert y(t)\Vert_U.
$$

We have the following proposition, based on a famous interpolation theorem, called the Riesz-Thorin's theorem.
\begin{proposition}
\label{prop:jauge}
Suppose that $\sigma$ is saturation as in Definition \ref{def:sat} with $\mathcal S:=U$. Therefore, given any $T>0$ and any $p>2$, any signal $y\in L^p_{\mathrm{loc}}([0,\infty);U)\cap \mathfrak U$ satisfies the following estimate
\begin{equation}
\label{eq:jauge-inequality-prop}
\int_0^T \mathcal{J}(\Vert y(t)\Vert_U) \diff t \geq \min\left(\Vert \mathbf P_T y\Vert_2^2,\frac{1}{2}\frac{\Vert \mathbf P_T y\Vert_{2}^{\frac{1}{\theta}}}{\Vert \mathbf P_T y\Vert^{\frac{1-\theta}{\theta}}_{p}}\right),
\end{equation}
with $\theta:=\frac{p-2}{2(p-1)}$.
\end{proposition}

\begin{proof}
Given any $T>0$, we denote by $\Omega_i$ (with $i\in\lbrace 1,2\rbrace$) the following subsets of the interval $[0,T]$
$$
\Omega_1:=\lbrace t\in [0,T]\mid \Vert y(t)\Vert_U\leq \delta\rbrace,\quad \Omega_2:=\lbrace t\in [0,T]\mid \Vert y(t)\Vert_U > \delta\rbrace
$$

By definition of $\mathcal J$, we have
\begin{equation}
\label{eq:decomp-jauge}
\int_0^T \mathcal{J}(\Vert y\Vert_U)\diff t=\int_{\Omega_1} \Vert y\Vert^2 \diff t + \int_{\Omega_2} \Vert y\Vert \diff t.
\end{equation}
Obviously, the following decomposition holds
\begin{equation}
\label{eq:decomp-obvious}
\int_0^T \Vert y\Vert^2_U \diff t = \int_{\Omega_1} \Vert y\Vert^2 \diff t + \int_{\Omega_2} \Vert y\Vert^2 \diff t
\end{equation}
For any signals $y\in L_{\mathrm{loc}}^p([0,+\infty);U)\cap \mathfrak U$, and for $i\in \lbrace 1,2\rbrace$, we use the notation
$$
\Vert y\Vert^p_{\Omega_i^p}:=\int_{\Omega_i} \Vert y(t)\Vert_U^p \diff t.
$$

Now, we are ready to prove the desired inequality \eqref{eq:jauge-inequality-prop}. To do so, we suppose that
$\Vert y\Vert^2_{\Omega_1^2} \geq \frac{1}{2} \Vert \mathbf P_T y\Vert_{2}^2$. Then, one has

\begin{equation}
\label{eq:obvious}
\int_0^T \mathcal J(\Vert y(t)\Vert_U) \diff t \geq \frac{1}{2}\Vert \mathbf P_T y\Vert_{2}^2,
\end{equation}
which proves one of the terms of the desired inequality.

Now, suppose that $\int_{\Omega_1} \Vert y(t)\Vert^2 \diff t < \frac{1}{2} \Vert \mathbf P_T y\Vert_{2}^2.$ Therefore, using \eqref{eq:decomp-obvious}, one has
$$
\Vert y\Vert_{\Omega_2^2}^2 =  \Vert \mathbf P_T y\Vert_2^2 - \Vert y\Vert_{\Omega_1^2}^2 >\frac{1}{2}\Vert \mathbf P_T y\Vert_2^2
$$

Moreover, invoking the Riesz-Thorin's theorem \cite[Theorem
1.1.1, Page 2]{bergh2012interpolation}, we obtain
$$
\Vert y\Vert_{\Omega_2^2}^2 \diff t \leq \Vert y\Vert_{\Omega_2^1}^{\theta} \Vert y\Vert^{1-\theta}_{\Omega_2^p},
$$
with $\theta:=\frac{p-2}{2(p-1)}$. Therefore, by assumption, we have
$$
\Vert y\Vert_{\Omega_2^1}\geq \frac{\Vert y\Vert_{\Omega_2^2}^{\frac{1}{\theta}}}{\Vert y\Vert^{\frac{1-\theta}{\theta}}_{\Omega_2^p}} >\frac{1}{2}\frac{\Vert \mathbf P_Ty\Vert_{2}^{\frac{1}{\theta}}}{\Vert \mathbf P_Ty\Vert^{\frac{1-\theta}{\theta}}_{p}}.
$$
Recall that \eqref{eq:decomp-jauge} holds. Therefore, we obtain that, for all $T>0$
$$
\int_0^T \mathcal J(\Vert y\Vert_U)\diff t >  \frac{1}{2}\frac{\Vert \mathbf P_Ty\Vert_{2}^{\frac{1}{\theta}}}{\Vert \mathbf P_Ty\Vert^{\frac{1-\theta}{\theta}}_{p}}.
$$
The latter result together with \eqref{eq:obvious} yields the desired inequality, which therefore concludes the proof.
\end{proof}

Our goal is to provide our first main result, which is based on Theorem \ref{thm:hastir}. We recall that it is extracted from \cite{hastir2025well}. Note that from Theorem \ref{thm:hastir}, we have an $L^\infty$-bound on the output $y$ and this $L^\infty$-bound depends on the initial state $z_0 \in D(A_\phi)$, i.e., for almost every $t\geq 0$,
\begin{equation}
\label{eq:pointwise-output}
\Vert y(t)\Vert_U \leq \Vert z_0\Vert_H + \Vert A_{\phi}(z_0)\Vert_H.
\end{equation}
The term appearing at the right hand side does not need to be a norm, because $A_\phi$ is a nonlinear operator, unless $B$ and $C$ are supposed to be bounded. Therefore, we have the following theorem
\begin{theorem}
\label{thm:SGES}
Consider \eqref{eq:cl-sat} with $v=\phi(y)$ and $d=r=0$. Suppose that the assumptions of Theorem \ref{thm:hastir} are satisfied, and for any $R>0$, suppose that $\Vert z_0\Vert_H + \Vert A_{\phi}(z_0)\Vert_H\leq R$. Recall that the output operator $C_\kappa$ is exactly observable in infinite-time for $\mathbb S$. Therefore, the origin of \eqref{eq:cl-sat} is SGES. 
\end{theorem}

\begin{proof}
In the \textbf{first step}, we apply Proposition \ref{prop:jauge} for $p=\infty$. We can check that $\theta=\frac{1}{2}$ in this case. We therefore have
$$
\int_0^T \mathcal J(\Vert y\Vert_U)\diff t \geq C_R \Vert \mathbf P_T y\Vert^2_2,
$$
where $C_R:=\min\left(1,\frac{1}{2R}\right).$

In a \textbf{second step}, we are going to translate the observability property of the linear system (i.e., \eqref{eq:cl-sat} with $v=0$) to the nonlinear system (when $v\neq 0$). Recall that, since $d=r=0$, we have
$$
\Vert z(t)\Vert_H^2 \leq \Vert z_0\Vert^2_H -\underline{c}_\sigma \int_0^T \mathcal J(\Vert y(t)\Vert_U)\diff t \leq \Vert z_0\Vert^2_H - C_R \Vert \mathbf P_T y\Vert^2_2.
$$
We know that $\tilde \Psi_\infty z_0$ is coercive on an infinite time horizon, since $C_\kappa$ is exactly observable in infinite time for $\mathbb S$. We want to translate the property of $\tilde \Psi_\infty$ (coming from the linear system) to $y$ (coming from the nonlinear system). Recall that
$$
\Psi_\infty z_0 = y(t) + \tilde{\mathbb F}_\infty \phi(\kappa y), 
$$
as stated in \eqref{eq:fixed-point1}. Since the system in \eqref{eq:cl-sat} is exponentially stable when $v=0$, the operator $\tilde{\mathbb F}_\infty$ is uniformly bounded in $L^2([0,\infty);U)$ as stated in \cite{weiss2000optimizability,weiss1994transfer}. Moreover, using the fact that $\phi$ is globally Lipschitz in $U$, we obtain:
$$
\Vert \mathbf P_T \Psi_\infty z_0\Vert_{\mathfrak U} \leq c_{\mathbb F} \Vert \mathbf P_T y\Vert_{\mathfrak U}.
$$
This together with the exact observability property implies that
\begin{equation}
\label{eq:observ-nl}
\Vert \mathbf P_T y\Vert_{\mathfrak U} \geq k_C \Vert z_0\Vert_H, 
\end{equation}
for all $T$ (note that $k_C$ depends on $T$). 

In the \textbf{third step}, we show the desired result. Denote by $E_n=\Vert z(nT)\Vert_H^2$. We can prove that
$$
E_{n+1} \leq (1-k_C^2 C_R) E_n.
$$
We know that, for a time $T$ large enough, we have $1-k_C^2C_R \in [0,1]$. It can be proved by induction that, for all $n\in \mathbb N$, $E_{n}\leq (1-k_C^2 C_R)^n E_0$. Since $(1-k_C^2 C_R)^n \leq e^{-\frac{k_C^2 C_R nT}{T}}$, we have $E_n \leq e^{-k_C^2 C_R n} E_0$. Recalling that $E_{n+1} \leq E_n$ (because the system is dissipative), we obtain that, for all $t\in [n,(n+1)T]$ $E(t)\leq E(nT)\leq e^{-\frac{k_C^2 C_R nT}{T}} E_0$. Since $t\leq (n+1)T$, we have $nT \geq t-T$. Therefore, we have, for all $t\in [nT,(n+1)T]$
\begin{equation}
\Vert z(t)\Vert^2_H \leq e^{k_C^2 C_R}e^{-\frac{k_C^2 C_R t}{T}} \Vert z_0\Vert^2_H,
\end{equation}
which proves the desired result, because $n$ and $T$ were chosen arbitrarily. 
\end{proof}

In the multi-dimensional case as explained in the example provided earlier, the coercivity with respect to one frequency $\lambda$ cannot hold in general, see, e.g., Remark \ref{rem:lack-coercivity}.  In general, $L^\infty$-bounds on the output is impossible. In light of Proposition \ref{prop:jauge}, one would like to obtain $L^p$-bounds. However, it is worth noting that $L^p$-spaces are not convenient if one is willing to use classical results in Hilbert spaces. As noted in \cite{paunonen2024admissibility}, a very useful space is $ H^\eta([0,\infty);U)$ with $\eta\in [0,1/2)$. By convention, $\eta=0$ corresponds to the $L^2$-space. This particular space can be defined as an interpolation of $H^1_0([0,\infty);U)$ and $L^2([0,\infty);U)$ as stated in \cite[Lemma A.10]{paunonen2024admissibility} (but it was already discussed in \cite{lions2012non} or \cite{di2012hitchhiker}). The space $H^\eta([0,\infty);U)$ is endowed with the norm

$$
\Vert u\Vert_{H^\eta}:=\Vert u\Vert_{L^2} + [u]_{H^\eta},
$$
where $[u]_{H^\eta}$ is the so-called \textit{semi-norm of Gagliardo}, defined as follows:
\begin{equation}
\label{def:semi-norm}
[u]^2_{H^\eta}:=\int_0^\infty \int_0^\infty \frac{\Vert u(t)-u(s)\Vert^2_U}{|t-s|^{1+2\eta}} \diff t\diff s,
\end{equation}
see \cite{di2012hitchhiker} for a nice introduction to fractional Sobolev spaces. The total norm of $H^\eta$ is known as the \textit{Sobolev--Slobodeckii norm}. We can define the same norm on a finite time-interval $[0,T]$ for any $T>0$ (at the price of adding the truncation operator, for instance). The notation will remain the same, except if the situation allows any confusion.

Another interpolated space of interest is the following one:
$$
H_\eta:=[H,D(A)]_{\eta}
$$
endowed with the norm:
$$
\Vert z\Vert_{\eta}:=\Vert z\Vert_{[H,D(A)]_\eta}.
$$
Because $A$ is densely defined in $H$, the interpolation can be done with methods described in \cite[Chapter 2 $\&$ 4]{bergh2012interpolation}. It is worth noting that, unless $A$ is sectorial, it is not possible to define it in a more abstract way.

Before considering the statement and the proof of our first main result, we will derive a regularity property (in the fractional space above mentioned) of the following fixed-point equation
\begin{equation}
\label{eq:fp-output}
y(t) = \tilde \Psi_\infty z_0 - \tilde{\mathbb{F}}_\infty(\phi(\kappa y(t))), 
\end{equation}
which comes from \eqref{eq:fixed-point} taking $d=r=0$. 

The following result shows that, as soon as $z_0\in H_\eta$, then the fractionnal regularity for the linear system \eqref{eq:cl-sat} (i.e., when $v$ is arbitrary) is conserved through the trajectories. 

\begin{lemma}
\label{lem:conserved-regularity}
Consider the system \eqref{eq:cl-sat} with $v\in \mathfrak U$ arbitrary, i.e., the system is linear. Let $T>0$ be an arbitrary positive constant. Suppose that $z_0\in H_\eta$ and $u\in H^\eta([0,T];U)$. Therefore $z\in C([0,T];H_\eta)$ and $y\in H^\eta([0,T];U)$.
\end{lemma}

\begin{proof}
In the \textbf{first step}, we show that, as soon as $z_0\in H_\eta$, then $\tilde{\Psi}_\infty z_0\in H^\eta([0,T];U)$. We already know that $\tilde\Psi_\infty$ is bounded in $L^2([0,T];U)$ for all $T>0$, since $\mathbb S$ is an exponentially stable semigroup. Suppose now that $z_0\in D(A_\kappa)$ (with $A_\kappa$ defined in \eqref{eq:new-generators}). Therefore, we have
$$
\frac{\diff}{\diff t} \Psi_\infty z_0=C_\kappa \frac{\diff}{\diff t}\mathbb S_t z_0=C_\kappa A_\kappa \mathbb S_t z_0=C_\kappa \mathbb S_t A_\kappa z_0.
$$
Therefore, it can be proved that 
$$
\Vert \Psi_\infty z_0\Vert_{H^1_0([0,T];U)} \leq K_{C_\kappa} \Vert z_0\Vert_{D(A_\kappa)},
$$
which shows that this operator is also bounded in $H^1_0([0,T];U)$. By interpolation, using \cite[Appendix A]{paunonen2024admissibility}, we can prove the desired result.

In the \textbf{second step}, given any $T>0$, we show that, for all $v\in H^\eta([0,T];U)$, $\tilde \Phi_t v\in H_\eta$ and $\tilde{\mathbb F}_{\infty} v\in H^{\eta}([0,T];U)$.

We first consider an input $v\in H^1_0([0,T];U)$. We denote by $v^\prime$ its weak derivative. Therefore, we can show that, for almost every $t\in [0,T]$

$$
\tilde\Phi_t v^\prime = \int_0^t \mathbb S_{t-s}B_\kappa v^\prime(s)\diff s. 
$$
Performing an integration by parts, we obtain
$$
\tilde \Phi_t v^\prime = B_\kappa v(t) + A_\kappa \tilde \Phi_t v=\frac{\diff}{\diff t} \tilde \Phi_t v,
$$
where, in the first equality, we have used the fact that $v(0)=0$ and in the last equality, we have used the strong form of \eqref{eq:cl-sat} when $z_0=0$, i.e., $z(t)= \Phi_t v$, for all $t\in [0,T]$. Using the infinite-time admissible of $B_\kappa$ for $\mathbb S$, it is clear that $\tilde \Phi_t v\in D(A_\kappa)$ for all $v\in H^1_0([0,T];U)$. By interpolation, we show that $\tilde \Phi_t v\in H_\eta$ for all $v\in H^\eta ([0,T];U)$. 

Recall that $\tilde {\mathbb F}_\infty$ is represented as follows $
\mathbb F_\infty v(t):=C_\Lambda \tilde\Phi_t v + Dv(t).$
Suppose that we consider the weak derivative of $v$. Then, for almost every $t\geq 0$
\begin{equation}
\label{eq:input-output-derivative}
\begin{aligned}
\tilde{\mathbb F}_\infty v^\prime(t)= C_\kappa \tilde \Phi_t v^\prime + D_\kappa v^\prime(t) = C_\kappa \frac{\diff}{\diff t} \tilde \Phi_t v^\prime + D_\kappa v^\prime(t) = \frac{\diff}{\diff t}(\tilde{\mathbb F}_\infty v(t)).
\end{aligned}
\end{equation}
With the same strategy than before, we can show that $\tilde{\mathbb F}_\infty$ is bounded from $H^1_0$ to $H^1_0$. Therefore, using the interpolation theory available in \cite[Appendix A]{paunonen2024admissibility}, we can prove $\tilde{\mathbb F}_\infty v\in H^\eta([0,T];U)$ for all $v\in H^\eta([0,T];U)$. Finally, since $H_\eta$ is invariant to the semigroup $\mathbb S$ (see e.g., \cite[Section V]{engel2000one}), the result is proved.\end{proof}

The following result demonstrates that the output $y$ inherits the regularity of $\tilde \Psi_\infty z_0$.
\begin{lemma}
\label{lem:l-nl}
Given any $T>0$, we have, for all $z_0\in H_\eta$, the function $t\mapsto \Psi_\infty z_0$ is in $H^\eta([0,T];U)$, with $0\leq \eta <\frac{1}{2}$. Therefore, the unique solution of \eqref{eq:fp-output} $y$ satisfies $y\in H^\eta([0,T];U)$.
\end{lemma}

\begin{proof}
Let $T>0$ and define the relation (i.e., an operator which does not need to be defined everywhere)

\[
\mathcal R:=\lbrace r\in \mathfrak U\mid \exists u\in\mathfrak U,\: \tilde{\mathbb F}_\infty r = u\rbrace.
\]
\begin{comment}
\[
\mathcal R:= \tilde{\mathbb F}_\infty^{-1}.
\]
Thus
\[
r\in \mathcal Ru
\quad\Longleftrightarrow\quad
\tilde{\mathbb F}_\infty r=u.
\]
\end{comment}
Because $\tilde{\mathbb F}_\infty$ is strictly output passive, this relation is maximal monotone. In the \textbf{first step}, we prove the result in $L^2([0,T];U)$ (therefore, for $H^\eta$ with $\eta=0$). Since
$\tilde{\mathbb F}_\infty$ is strictly output passive, for all
$r_1\in \mathcal Ru_1$ and $r_2\in \mathcal Ru_2$ one has
\[
\langle r_1-r_2,u_1-u_2\rangle_{L^2}
\ge
\mu\|u_1-u_2\|_{L^2}^2 .
\]
Moreover, since $\phi$ is satifies \eqref{eq:strict-output}, we have

\[
\langle r_1-r_2+\phi(u_1)-\phi(u_2),u_1-u_2\rangle_{L^2}
\ge
\mu\|u_1-u_2\|_{L^2}^2 .
\]
Assuming that $\mathcal R$ is maximal monotone, Rockafellar's theorem implies that
$\mathcal R+\phi$ is maximal monotone. Since it is strongly monotone, its inverse
is single-valued and Lipschitz:
\[
\|(\mathcal R+\phi)^{-1}g_1-(\mathcal R+\phi)^{-1}g_2\|_{L^2}
\le
\frac1\mu \|g_1-g_2\|_{L^2}.
\]

In the \textbf{second step}, we prove the fractional regularity estimate. Let
\[
r+\phi(u)=g,\qquad r\in \mathcal Ru,
\]
with $g=\Psi_\infty z_0$. For $h>0$, denote by $S_h$ the right-shift operator,
\[
(S_h f)(t)=
\begin{cases}
0, & 0<t<h,\\
f(t-h), & t\ge h.
\end{cases}
\]
Since the graph of $\mathcal R=\tilde{\mathbb F}_\infty^{-1}$ is shift-invariant,
\[
r\in \mathcal Ru
\quad\Longrightarrow\quad
S_h r\in \mathcal R(S_hu).
\]
Assuming $\phi(0)=0$, the shifted equation reads
\[
S_h r+\phi(S_hu)=S_hg.
\]
Applying the $L^2$ Lipschitz estimate for $(\mathcal R+\phi)^{-1}$ to the two
right-hand sides $S_hg$ and $g$, we obtain
\[
\|S_hu-u\|_{L^2([0,T];U)}
\le
\frac1\mu
\|S_hg-g\|_{L^2([0,T];U)} .
\]
Squaring, multiplying by $h^{-1-2\eta}$ and integrating in $h$ gives
\[
\int_0^T
\frac{\|S_hu-u\|_{L^2(0,T;U)}^2}{h^{1+2\eta}}\,\diff h
\le
\frac1{\mu^2}
\int_0^T
\frac{\|S_hg-g\|_{L^2(0,T;U)}^2}{h^{1+2\eta}}\,\diff h .
\]
For $0<\eta<1/2$, the right-hand side is bounded by the
$H^\eta([0,T];U)$ norm of $g$, because the zero-extension characterization of
$H^\eta$ is equivalent to the usual Sobolev--Slobodeckii norm. Hence
\[
[u]_{H^\eta([0,T];U)}
\le
\frac{1}{\mu}
[g]_{H^\eta([0,T];U)}.
\]
Together with the $L^2$ estimate, this yields
\[
\|u\|_{H^\eta([0,T];U)}
\le
\max\left(1,\frac{1}{\mu}\right)
\|g\|_{H^\eta([0,T];U)}.
\]
Recalling that $g=\Psi_\infty z_0$ concludes the proof.
\end{proof}

Now, we state and prove that the operator $\sigma$ is bounded from $H^\eta$ to $H^\eta$.

\begin{lemma}
\label{lem:bound-sigma-eta}
For any signals $y\in H^\eta_{\mathrm{loc}}([0,\infty);U)$, we have
$$
\Vert \sigma(y)\Vert_{H^\eta} \leq \Vert y\Vert_{H^{\eta}}.
$$
\end{lemma}

\begin{proof}
Let $T>0$ be arbitrary. Recall that $\sigma$, as an incrementally strictly output passive operator, is globally Lipschitz with Lipschitz constant equal to $1$. It is moreover a local operator. Therefore, we have
\begin{equation}
[\sigma(y)]_{H^\eta}=\int_0^{T} \int_0^T \frac{\Vert \sigma(y(t))-\sigma(y(s))\Vert^2_U}{|t-s|^{1+2\eta}} \diff t \diff s \leq \int_0^T \int_0^T \frac{\Vert y(t)-y(s)\Vert_U^2}{|t-s|^{1+2\eta}} \diff t \diff s,
\end{equation}
Therefore, by definition of the $H^\eta$-norm, we obtain the desired result.
\end{proof}

\begin{theorem}
\label{thm:SUGAS}
Suppose that $\mathcal S=U$ and that Assumptions \ref{ass:impedance}, \ref{ass:regular} and \ref{ass:stab} are satisfied. Therefore, the origin of \eqref{eq:cl-sat} with $u=d=0$ is SUGAS with $H_s:=H_\eta$. 
\end{theorem}
\begin{proof}
Let \(T>0\) and \(R>0\) be fixed. Throughout the proof, \(M_*>0\) denotes a generic positive constant whose value may change from line to line and which depends only on \(R\), \(T\), the parameters of the saturation map, and the admissibility, regularity and observability constants of the associated linear system. Similarly, \(m_*>0\) denotes a generic positive constant with the same type of dependence. In particular, these constants are independent of the iteration index introduced below.

In the \textbf{first step}, we describe precisely the nonlinear dissipation on one time interval. Let \(z_0\in H_\eta\) be such that
\[
        \|z_0\|_{H_\eta}\le R.
\]
By Lemma~\ref{lem:conserved-regularity},
\[
        \|\mathbf P_T\Psi_\infty z_0\|_{H^\eta(0,T;U)}
        \le M_* .
\]
Hence, for every
\[
        2<p<\frac{2}{1-2\eta},
\]
the Sobolev embedding \(H^\eta(0,T;U)\hookrightarrow L^p(0,T;U)\) gives
\begin{equation}
\label{eq:free-output-Lp-clean}
        \|\mathbf P_T\Psi_\infty z_0\|_{L^p(0,T;U)}
        \le M_* .
\end{equation}

We introduce
\[
        \mathcal A_T(z_0)
        :=
        \int_0^T
        \J\bigl(\|\mathbf P_T\Psi_\infty z_0(t)\|_U\bigr)\,\diff t,
\]
and
\[
        \mathcal D_T(z_0)
        :=
        \int_0^T
        \J\bigl(\|y(t)\|_U\bigr)\,\diff t,
\]
where \(y\) denotes the nonlinear closed-loop output on \([0,T]\) with the initial state \(z_0\).

The input-output relation on \([0,T]\) is
\begin{equation}
\label{eq:io-relation-clean}
        \mathbf P_T\Psi_\infty z_0
        =
        y+\widetilde{\mathbb F}_T\phi(y).
\end{equation}
By the subadditivity estimate for \(\J\), there exists \(M_*>0\) such that
\begin{equation}
\label{eq:J-subadditive-clean}
        \mathcal A_T(z_0)
        \le
        M_*\mathcal D_T(z_0)
        +
        M_*
        \int_0^T
        \J\bigl(\|\widetilde{\mathbb F}_T\phi(y)(t)\|_U\bigr)\,\diff t .
\end{equation}
Since \(\J(r)\le r^2\) and \(\widetilde{\mathbb F}_T\) is bounded on \(L^2(0,T;U)\),
\begin{equation}
\label{eq:Fphi-L2-clean}
        \int_0^T
        \J\bigl(\|\widetilde{\mathbb F}_T\phi(y)(t)\|_U\bigr)\,\diff t
        \le
        M_*\|\phi(y)\|_{L^2(0,T;U)}^2 .
\end{equation}

Let
\[
        2<q<\frac{2}{1-2\eta},
        \qquad
        \theta_1:=\frac{q-2}{2(q-1)}\in(0,1/2).
\]
Interpolating between \(L^1(0,T;U)\) and \(L^q(0,T;U)\), we get
\[
        \|\phi(y)\|_{L^2(0,T;U)}
        \le
        \|\phi(y)\|_{L^1(0,T;U)}^{\theta_1}
        \|\phi(y)\|_{L^q(0,T;U)}^{1-\theta_1}.
\]
By the Sobolev embedding and Lemma~\ref{lem:bound-sigma-eta},
\[
        \|\phi(y)\|_{L^q(0,T;U)}
        \le
        M_*.
\]
Moreover, by \eqref{eq:key-inequality} and \eqref{eq:jauge},
\[
        \|\phi(y)\|_{L^1(0,T;U)}
        \le
        \int_0^T \langle \sigma(y(t)),y(t)\rangle_U\,\diff t
        \le
        M_*\mathcal D_T(z_0).
\]
Therefore,
\[
        \|\phi(y)\|_{L^2(0,T;U)}^2
        \le
        M_*\mathcal D_T(z_0)^{2\theta_1}.
\]
Together with \eqref{eq:Fphi-L2-clean}, this gives
\begin{equation}
\label{eq:Fphi-DT-clean}
        \int_0^T
        \J\bigl(\|\widetilde{\mathbb F}_T\phi(y)(t)\|_U\bigr)\,\diff t
        \le
        M_*\mathcal D_T(z_0)^{2\theta_1}.
\end{equation}
Thus, from \eqref{eq:J-subadditive-clean},
\begin{equation}
\label{eq:AT-DT-clean}
        \mathcal A_T(z_0)
        \le
        M_*
        \left(
        \mathcal D_T(z_0)
        +
        \mathcal D_T(z_0)^{2\theta_1}
        \right).
\end{equation}

By passivity and \eqref{eq:jauge}, the nonlinear dissipation is bounded on bounded subsets of \(H_\eta\):
\[
        \mathcal D_T(z_0)
        \le
        M_*\|z_0\|_H^2
        \le
        M_*.
\]
Since \(0<2\theta_1<1\), we infer from \eqref{eq:AT-DT-clean} that
\begin{equation}
\label{eq:AT-DT-absorbed-clean}
        \mathcal A_T(z_0)
        \le
        M_*\mathcal D_T(z_0)^{2\theta_1}.
\end{equation}
Consequently,
\begin{equation}
\label{eq:DT-AT-clean}
        \mathcal D_T(z_0)
        \ge
        m_*
        \mathcal A_T(z_0)^{\frac1{2\theta_1}}.
\end{equation}

It remains to estimate \(\mathcal A_T(z_0)\) from below. Applying Proposition~\ref{prop:jauge} to the signal \(\mathbf P_T\Psi_\infty z_0\), using \eqref{eq:free-output-Lp-clean} and exact observability, we obtain
\[
        \mathcal A_T(z_0)
        \ge
        m_*
        \|z_0\|_H^{\frac1{\theta_2}},
        \qquad
        \theta_2:=\frac{p-2}{2(p-1)}.
\]
Combining this inequality with \eqref{eq:DT-AT-clean}, we obtain
\begin{equation}
\label{eq:DT-z0-clean}
        \mathcal D_T(z_0)
        \ge
        m_*
        \|z_0\|_H^{\frac1{2\theta_1\theta_2}}.
\end{equation}
Equivalently, if \(E_0:=\|z_0\|_H^2\), then
\begin{equation}
\label{eq:DT-E0-clean}
        \mathcal D_T(z_0)
        \ge
        m_* E_0^m,
        \qquad
        m:=\frac1{4\theta_1\theta_2}>1.
\end{equation}

In a \textbf{second step}, we derive some discrete energy estimate. For \(n\in\mathbb N\), set
\[
        E_n:=\|z(nT)\|_H^2.
\]
By dissipativity, \(E_{n+1}\le E_n\). Moreover, the \(H_\eta\)-regularity is propagated along the trajectory by Lemma~\ref{lem:conserved-regularity}; hence the constants \(M_*,m_*\) above can be chosen uniformly with respect to \(n\), provided \(\|z_0\|_{H_\eta}\le R\).

Applying Step~1 on the interval \([nT,(n+1)T]\), with initial state \(z(nT)\), yields
\[
        \int_{nT}^{(n+1)T}
        \J\bigl(\|y(t)\|_U\bigr)\,\diff t
        \ge
        m_* E_n^m.
\]
On the other hand, the energy inequality gives
\[
        E_{n+1}-E_n
        \le
        -2\underline c_\sigma
        \int_{nT}^{(n+1)T}
        \J\bigl(\|y(t)\|_U\bigr)\,\diff t.
\]
Therefore, after possibly modifying \(m_*>0\),
\begin{equation}
\label{eq:discrete-energy-clean}
        E_{n+1}-E_n
        \le
        -m_*E_n^m,
        \qquad n\in\mathbb N.
\end{equation}
Since
\[
        \theta_1=\frac{q-2}{2(q-1)},
        \qquad
        \theta_2=\frac{p-2}{2(p-1)},
\]
we have
\[
        m
        =
        \frac1{4\theta_1\theta_2}
        =
        \frac{(q-1)(p-1)}{(q-2)(p-2)}
        >1.
\]
Set
\[
        \nu:=m-1>0.
\]

In a \textbf{third step}, we prove the polynomial stability result. The discrete estimate \eqref{eq:discrete-energy-clean} reads
\[
        E_{n+1}-E_n
        \le
        -m_*E_n^{1+\nu}.
\]
Since \(E_{n+1}\le E_n\), the convexity of \(s\mapsto s^{-\nu}\) on \((0,\infty)\) gives
\[
        E_{n+1}^{-\nu}-E_n^{-\nu}
        \ge
        \nu E_n^{-\nu-1}(E_n-E_{n+1}).
\]
Thus
\[
        E_{n+1}^{-\nu}-E_n^{-\nu}
        \ge
        \nu m_*.
\]
Summing from \(0\) to \(n-1\), we obtain
\[
        E_n^{-\nu}
        \ge
        E_0^{-\nu}+\nu m_* n.
\]
Hence
\[
        E_n
        \le
        \left(E_0^{-\nu}+\nu m_*n\right)^{-1/\nu}
        \le
        M_*(1+n)^{-1/\nu}.
\]
Finally, if \(t\in[nT,(n+1)T]\), dissipativity gives
\[
        \|z(t)\|_H^2\le E_n.
\]
Since \(n\ge t/T-1\), we infer, after modifying \(M_*\), that
\[
        \|z(t)\|_H^2
        \le
        M_*(1+t)^{-1/\nu},
        \qquad t\ge0.
\]
Equivalently,
\[
        \|z(t)\|_H
        \le
        M_*(1+t)^{-\frac1{2\nu}},
        \qquad t\ge0.
\]
This proves the semi-uniform polynomial decay estimate on bounded subsets of \(H_\eta\).
\end{proof}

Now we state our robust result, namely, the fact that the system \eqref{eq:cl-sat} is semi-uniformly ISS.

\begin{theorem}
\label{thm:ISS}
Suppose that the properties assumed in Theorem \ref{thm:SUGAS} are satisfied. Suppose that $d\in \mathfrak U$ is uniformly bounded in $\mathfrak U$, i.e., $\Vert d(t)\Vert_{U}\leq K_d$ for almost every $t\geq 0$ and with $K_d>0$. Then, the trajectories of \eqref{eq:cl-sat} are semi-uniformly ISS.
\end{theorem}

%\textcolor{red}{Réécrire}

\begin{proof}
Let $R>0$ and let $z_0\in H_\eta$ be such that
$\|z_0\|_{H_\eta}\le R$. We denote by $(z,y,e)$ the solution of \eqref{eq:cl-sat}
associated with the initial condition $z_0$ and the inputs $(d,r)$.
We also denote by $(\bar z,\bar y,\bar e)$ the solution associated with
the same initial condition $z_0$ and with $d=r=0$.

Set
\[
\delta z:=z-\bar z,\qquad
\delta y:=y-\bar y.
\]
It is worth noting that this is not a superposition principle for the nonlinear system; it is
only the difference between two trajectories.

By Theorem \ref{thm:SUGAS}, there exists
$\rho\in\mathcal{KL}$ such that
\[
\|\bar z(t)\|_H
\le
\rho(\|z_0\|_{H_\eta},t),
\qquad t\ge0.
\]

We now estimate $\delta z$. Since the two trajectories have the same
initial condition, the initial difference is zero. Using the incremental
dissipation estimate, we obtain
\[
\begin{aligned}
\|\delta z(t)\|_H^2
\leq&
2\int_0^t \langle d(s),\delta y(s)\rangle_U\,\diff s
-\mu\|\delta y\|^2_{L^2(0,t;U)}
\\
&+
2\int_0^t
\left\langle
\delta y(s),
\sigma(y(s)+r(s))-\sigma(y(s))
\right\rangle_U\,\diff s .
\end{aligned}
\]
The first term is treated thanks to the uniform bound on $d$:
\[
2\int_0^t \langle d(s),\delta y(s)\rangle_U\,ds
-\frac{\mu}{2}\|\delta y\|^2_{L^2(0,t;U)}
\le \left(K_d-\frac{\mu}{2} \right) \Vert \delta y\Vert^2_{L^2([0,t];U)}.
\]
We select $K_d<\frac{\mu}{2}$ so that this term becomes negligible.

For the term containing $r$, we use the structural saturation estimate
\eqref{eq:ineq-sat-ISS}. More precisely,
\[
\left\langle
\delta y,
\sigma(y+r)-\sigma(y)
\right\rangle_U
\le
K_\sigma\|r\|_{\mathcal S},
\]
and therefore
\[
2\int_0^t
\left\langle
\delta y(s),
\sigma(y(s)+r(s))-\sigma(y(s))
\right\rangle_U\,\diff s
\le
2K_\sigma\|\mathbf P_t r\|_{\mathfrak R}.
\]

Consequently,
\[
\|\delta z(t)\|_H^2
\le
2K_\sigma\|\mathbf P_t r\|_{\mathfrak R}.
\]
Thus
\[
\|\delta z(t)\|_H
\le \sqrt{2K_\sigma}\|\mathbf P_t r\|_{\mathfrak R}^{1/2}.
\]

Since
\[
z(t)=\bar z(t)+\delta z(t),
\]
we get
\[
\|z(t)\|_H
\le
\rho(\|z_0\|_{H_\eta},t)
+\sqrt{2K_\sigma}\|\mathbf P_t r\|_{\mathfrak R}^{1/2}.
\]
This is the desired semi-uniform ISS estimate.
\end{proof}

\begin{remark}
The result of Theorems \ref{thm:ISS} and \ref{thm:SUGAS} can be applied for \eqref{eq:wave} in Example \ref{example:wave} as soon as a Geometric Control Condition is satisfied (see e.g., \cite{bardos1992sharp} for more details. For the 1D version of the wave equation, Theorem \ref{thm:SGES} can be applied, since the coercive condition has been proved for this case in Remark \ref{rem:lack-coercivity}. 
\end{remark}

\section{Conclusion}

In this paper, we have proved that, if the initial state belongs to some fractional Sobolev space, then both the linear and the nonlinear systems share a similar fractional regularity. This regularity analysis was needed to provide an asymptotic analysis of the trajectories, without or with disturbances.

The central message of the paper is therefore that interpolation theory provides a natural bridge between quadratic observability estimates and nonlinear dissipative feedback laws. This bridge allows one to obtain stability estimates for a large class of impedance passive systems without relying on PDE-specific multiplier techniques or on an explicit characterization of the nonlinear generator.

Many questions remain open. We state some of them:

\paragraph{Question 1:} \textit{Most of the results stated in this article are proved in the situation where $\mathcal S=U$. Is it possible to obtain results in the case where $\mathcal S\neq U$ ?}

The present work focuses on nonlinearities acting on the whole Hilbert space U. An interesting extension would be to replace the maps $\sigma$ by metric projections onto closed convex subsets of $U$ (namely, $\mathcal S$), or more generally by resolvents of maximal monotone operators. Recent developments on projected incrementally scattering passive systems \cite{singh2025projected} suggest that such an extension may be possible. Whether the fractional regularity preservation established here remains valid in this more general setting is left for future work.

\paragraph{Question 2:} \textit{Is it possible to use such techniques for other kinds of nonlinearities ? More precisely, the case of non-smooth feedback could be interesting to ensure asymptotic stability despite disturbance ?}

The question is closely related to the sliding mode control methodology, for which only a few extensions are available \cite{balogoun2022sliding}. Nevertheless, we suspect that, without changing the output space, the stability would not be global, but rather practical, as illustrated in \cite{chitour2020one,Xu2019Saturated}.

\paragraph{Question 3:} \textit{Is it possible to derive an integral ISS estimate with respect to $d$ thanks to a Lyapunov functional ?}

We have seen that the estimate given in Theorem \ref{thm:ISS} does not depend on the disturbance $d$, because $d$ is supposed to be uniformly bounded. In finite-dimension, this result has been improved by proving that the system is integral ISS with respect to $d$.\\ 

\textbf{Acknowledgment:} The author would like to thank Sylvain Ervedoza for the interpolation trick used all along the article. He would also thank Nicolas Vanspranghe for having been kind enough to introduce him to the concept of fractional Sobolev spaces (through the article \cite{paunonen2024admissibility} and also by discussing with me) and to provide me some important references about stabilization of infinite-dimensional systems with a nonlinear damping.

\label{sec:conclusion}

\bibliographystyle{plain}
\bibliography{bibsm}

%%%%%%%%%%%%%%%%%%%%%%%%%%%%%%%%%%%%%%%%%%%%%%%%%%%%%%%%%%%%%%%%%%%%%%%%%%%%%%%%%%%%%%%%%%%%%%%%%%%%%%%%%%%%%%%%%%%%%%%%%%%%%%%

\end{document}